\documentclass[11pt,a4paper]{amsart}
\usepackage{amsfonts,amsmath,amssymb,amsthm,url}
\usepackage[utf8]{inputenc}
\usepackage[T1]{fontenc}
\usepackage{booktabs}
\usepackage{float}
\usepackage{graphicx}
\usepackage[arrow,cmtip,matrix]{xy}
\usepackage{latexsym}
\usepackage{mathtools}
\usepackage[pagebackref]{hyperref}

\DeclareMathOperator{\Aut}{Aut}
\DeclareMathOperator{\Cl}{Cl}
\DeclareMathOperator{\ord}{ord}
\DeclareMathOperator{\pr}{pr}
\DeclareMathOperator{\vol}{Vol}

\newtheorem{theorem}{Theorem}[section]
\newtheorem{proposition}[theorem]{Proposition}
\newtheorem{lemma}[theorem]{Lemma}
\newtheorem{corollary}[theorem]{Corollary}

\theoremstyle{definition}  
\newtheorem{definition}[theorem]{Definition}
\newtheorem*{notation}{Notation}  

\theoremstyle{remark}  
\newtheorem{remark}[theorem]{Remark}
\newtheorem{example}[theorem]{Example}

\DeclarePairedDelimiter\floor{\lfloor}{\rfloor}

\newcommand{\A}{\mathbb{A}}
\newcommand{\Gm}{\mathbb{G}_{\mathrm{m}}}
\newcommand{\K}{\mathbb{K}}

\renewcommand{\P}{\mathbb{P}}
\newcommand{\Q}{\mathbb{Q}}
\newcommand{\R}{\mathbb{R}}
\newcommand{\Z}{\mathbb{Z}}

\newcommand{\cD}{\mathcal{D}}
\newcommand{\cH}{\mathcal{H}}
\newcommand{\cJ}{\mathcal{J}}
\newcommand{\cN}{\mathcal{N}}
\newcommand{\cO}{\mathcal{O}}
\newcommand{\cS}{\mathcal{S}}
\newcommand{\cU}{\mathcal{U}}

\newcommand{\fB}{\mathfrak{B}}
\newcommand{\fD}{\mathfrak{D}}
\newcommand{\fI}{\mathfrak{I}}
\newcommand{\fP}{\mathfrak{P}}
\newcommand{\fS}{\mathfrak{S}}

\newcommand{\fa}{\mathfrak{a}}
\newcommand{\fb}{\mathfrak{b}}
\newcommand{\fc}{\mathfrak{c}}
\newcommand{\fd}{\mathfrak{d}}
\newcommand{\fe}{\mathfrak{e}}
\newcommand{\fm}{\mathfrak{m}}
\newcommand{\fp}{\mathfrak{p}}

\newcommand{\fr}{\mathfrak{r}}

\newcommand\isom{\stackrel{\sim}{\longrightarrow}}

\title[Rational points on weighted projective spaces]{Counting rational points on weighted projective spaces over number fields}

\author{Peter Bruin}
\address{Mathematisch Instituut, Universiteit Leiden, Postbus 9512, 2300 RA Leiden, Netherlands}
\email{P.J.Bruin@math.leidenuniv.nl}
\thanks{The first-named author was partially supported by the Dutch Research Council (NWO/OCW), as part of the Quantum Software Consortium programme (project number 024.003.037).}

\author{Irati Manterola Ayala}
\address{Simula UiB, Thormøhlens gate 53D, N-5006 Bergen, Norway}
\email{irati@simula.no}

\begin{document}
	\maketitle
	
	\begin{abstract}
		Deng \cite{De} gave an asymptotic formula for the number of rational points on a weighted projective space over a number field with respect to a certain height function. We prove a generalization of Deng's result involving a morphism between weighted projective spaces, allowing us to count rational points whose image under this morphism has bounded height. This method provides a more general and simpler proof for a result of the first-named author and Najman \cite{BrNa} on counting elliptic curves with prescribed level structures over number fields. We further include some examples of applications to modular curves.
	\end{abstract}
	
	\section{Introduction}
        \label{sec1}
	
	In 2013 R. Harron and A. Snowden \cite{HaSn} gave an asymptotic expression for the number of elliptic curves over $\Q$ of bounded height whose torsion subgroup is isomorphic to one of the possible torsion groups over $\Q$. This result was generalized by the first-named author and Najman \cite{BrNa} to all number fields and all level structures~$G$ such that the modular curve $X_G$ is a weighted projective line and the morphism $X_G\to X(1)$ satisfies a certain condition. The main result is stated in \cite[Theorem 7.6]{BrNa}, which gives asymptotic lower and upper bounds for the number of elliptic curves over a number field with prescribed level structures and bounded height.
	
	The main motivation for this work is to provide the tools for both generalizing and simplifying the proof of \cite[Theorem 7.6]{BrNa}. That proof uses a result of Deng \cite[Theorem (A)]{De}, which gives an asymptotic formula for the number of rational points on weighted projective spaces over number fields with bounded height (called \textit{size} by Deng), generalizing a result of Schanuel \cite{Sc}. Namely, one has 
	\begin{equation}
		\label{equ}
		\#\{x\in\P(w)(\K)\mid S_{w}(x)\leq T\}\sim CT^{|w|},
	\end{equation}
	where $C$ is a constant only depending on the choice of $\K$ and $w$. However, the main result in \cite{BrNa} requires an expression for $S_{u}(\phi(x))$ for a morphism $\phi:\P(w)_{\K}\to \P(u)_{\K}$ between weighted projective spaces; see \cite[Corollary~6.5]{BrNa}. In view of this, it is natural to seek a generalization of equation~\eqref{equ} involving the inequality $S_{u}(\phi(x))\leq T$ instead of $S_{w}(x)\leq T$.

	The aim of this paper is to generalize the results obtained in \cite{De}, in order to determine an asymptotic formula for the quantity $\#\{x\in\P(w)(\K)\mid S_{u}(\phi(x))\leq T\}$ as $T\to\infty$. We prove in our main result (Theorem~\ref{maintheorem}) that this quantity can be expressed in a similar way to equation~\eqref{equ} as
	\begin{equation*}
		\#\{x\in\P(w)(\K)\mid S_{u}(\phi(x))\leq T\}\sim CT^{|w|/e},
	\end{equation*}
	where now the constant also depends on~$\phi$. The strategy of this paper is analogous to that of \cite{De}, but the technical details are more involved.
        As a consequence, \cite[Theorem 7.6]{BrNa} can be sharpened to an asymptotic formula for the number of elliptic curves over a number field with prescribed level structure and bounded height, rather than just asymptotic upper and lower bounds. Including the improved version of the theorem is beyond the scope of this paper, but it is a fairly straightforward consequence of our result.
	
	The paper is organized as follows. In this introductory section we collect definitions and results on weighted projective spaces, morphisms between them and height functions, and we recall some results from algebraic number theory. We also fix the notation we will use throughout the paper.
In Section~\ref{sec2} we introduce various technical objects used in the proof of the main result and study some of their properties. In Section~\ref{sec3} we study our main counting problem using a method similar to that of~\cite{De}, and we state and prove our main result, Theorem~\ref{maintheorem}. In Section~\ref{sec4} we apply our counting result to two examples related to counting elliptic curves with bounded height and prescribed level structure. Finally, in Section~\ref{sec5} we discuss related results and possible directions for future work.
	
	We further note that although the application to counting elliptic curves is our main motivation, we expect that this approach can be applied to other counting problems.
	
	\subsection{Weighted projective spaces}
	
	\begin{definition}
		Let $w=(w_{1},\dots,w_{m})$ be an $m$-tuple of positive integers, with $m\geq 1$. Then the \textit{weighted projective space with weight $w$} is the quotient stack
		\[
		\P(w) := [\Gm\backslash\A^m_{\ne0}]
		\]
		over $\Z$, where $\A^m_{\ne0}$ is the complement of the zero section in the $m$-dimensional affine space~$\A^m$ and the multiplicative group $\Gm$ acts on $\A^m_{\ne0}$ by the \emph{weight~$w$ action}
		\[
		(\lambda, (x_{1},...,x_{m})) \mapsto \lambda_{*}(x_{1},...,x_{m}) :=
		(\lambda^{w_{1}}x_{1},...,\lambda^{w_{m}}x_{m}).
       \]
	\end{definition}

    If $\K$ is a field, the set of $\K$-points of $\P(w)$ is
	\begin{equation*}
		\P(w)(\K):=\K^{\times}\backslash(\K^{m}-\{0\}),
	\end{equation*}
	where the action is given by the same formula as above.

    \begin{lemma}
		\label{morphism}
		Let $\K$ be a field, let $m\ge2$, and let $w=(w_{1},\dots,w_{m})$ and $u=(u_{1},\dots,u_{m})$ be two $m$-tuples of positive integers. We view $\K[x_{1},\dots,x_{m}]$ as a graded $\K$-algebra with $x_{i}$ homogeneous of degree $w_{i}$ for all $i$. Let $P_{u,w}(\K)$ be the set of $m$-tuples $(f_{1},\dots,f_{m})$ of polynomials in $\K[x_{1},\dots,x_{m}]$ with the following properties:
		\begin{enumerate}
			\item There exists $e=e(f_{1},\dots,f_{m})\in \Z_{\geq 0}$ for which $f_{i}$ is homogeneous of degree $eu_{i}$ for all $i$.
			\item The homogeneous ideal $\sqrt{(f_{1},\dots,f_{m})}\subseteq \K[x_{1},\dots,x_{m}]$ contains the ideal $(x_{1},\dots,x_{m})$.
		\end{enumerate}
		Let $\K^\times$ act on $P_{u,w}(\K)$ by $c(f_1,\ldots,f_m)=(c^{u_1}f_1,\ldots,c^{u_m}f_m)$. Then there is a canonical bijection from $P_{u,w}(\K)$ to the set of morphisms $\P(w)_{\K}\to\P(u)_{\K}$ sending the class of $(f_1,\ldots,f_m)$ to the morphism $\phi$ given by $\phi(x_1,\ldots,x_m)=(f_1(x_1,\ldots,x_m),\ldots, f_m(x_1,\ldots,x_m))$.
	\end{lemma}
	
	\begin{proof}
		See \cite[Lemma 4.1]{BrNa} for a proof for $m=2$; the same proof works for general $m\ge2$.
	\end{proof}
	
	\subsection{Algebraic number theory}
        \label{ant}
	
	Let $\K$ denote a number field with ring of integers $\cO_{\K}$. We let $\fp$ range over integral prime ideals of $\K$, and in the case $\K=\Q$, we write $\fp=p$ for $p$ a prime number. 
	
	We recall the definitions of the Möbius function and the Dedekind zeta function of a number field. 
	
	\begin{definition}
	    We define the \textit{Möbius function of $\K$} to be 
		\begin{equation*}
		    \begin{aligned}
			\mu_{\K}: \{\text{integral ideals of }\K\} &\longrightarrow \{0,\pm 1\} \\
			I &\longmapsto \begin{cases}
				0, \text{ if } I\subset\fp^{2} \text{ for some } \fp;\ \\
				(-1)^{k},\text{ if } I=\fp_{1}\cdots\fp_{k},
			\end{cases}
		\end{aligned}
		\end{equation*}
		where $\fp_{1},\dots,\fp_{k}$ denote distinct prime ideals of $\K$.
	\end{definition}
	
	\begin{definition}
		We define the \textit{Dedekind zeta function of $\K$} to be 
		\begin{equation*}
		    \begin{aligned}
			\zeta_{\K}: \{t\in\mathbb{C}\mid \Re(t)>1\} &\longrightarrow \mathbb{C} \\
			t &\longmapsto \sum_{I\neq 0}(N(I))^{-t}=\prod_{\fp}(1-(N(\fp))^{-t})^{-1},
		\end{aligned}
		\end{equation*}
		where the sum ranges over all non-zero ideals $I$ of $\cO_{\K}$.
	\end{definition}

	\begin{notation}
		Let $I\subset\K$ a non-zero ideal. We denote by $\ord_{\fp}(I)$ the order of~$I$ at~$\fp$, where $I$ factors uniquely as
		\begin{equation*}
			I=\prod_{\fp}\fp^{\ord_{\fp}(I)}.
		\end{equation*}
		For $x\in\Q-\{0\}$, $\ord_{p}(x)$ denotes the order of $x$ at $p$, where 
		\begin{equation*}
			x=\prod_{p}p^{\ord_{p}(x)}.
		\end{equation*}
	\end{notation}
	
        Let $\Omega_{\K}$ denote the set of infinite places of $\K$.
	For every place $v$ of~$\K$, we normalize the $v$-adic norm as follows: for $v\in\Omega_{\K}$ we put $|x|_{v}=|x|^{N_{v}}$ where $N_{v}=1$ (resp.\ 2) if $v$ is real (resp.\ complex), and for finite places we put
	\begin{equation*}
		|x|_{\fp}=N(\fp)^{-\ord_{\fp}(x)}.
	\end{equation*}
	Note that $[\K:\Q]=r_{1}+2r_{2}$ where $r_{1}$ denotes the number of real places and $r_{2}$ the number of complex places.

	\subsection{The height function}
	
	We define the scaling ideal and the height function as in \cite[Definition 3.2]{De}. Let $w=(w_{1},\dots,w_{m})$ be a tuple of positive integers.
	
	\begin{definition}
          \label{scalingideal}
		Let $x\in \K^{m}-\{0\}$. We define the \textit{scaling ideal of $x$ with respect to $w$} to be $\fI_{w}(x)$, where 
		\begin{equation*}
			\fI_{w}^{-1}(x):=\{a\in \K : a_{*}x\in \cO_{\K}\}
		\end{equation*}
		and $a$ acts on $x$ with the weight $w$ action.
	\end{definition}
	
	\begin{lemma}
		\label{scalingidealother}
		The scaling ideal satisfies the following properties
		\begin{enumerate}
			\item $\fI_{w}(a_{*}x)=a\fI_{w}(x)$ for $a\in\K^{\times}$. Moreover, $\fI_{w}(x)$ is an integral ideal of $\cO_{\K}$ for $x=(x_{1},\dots,x_{m})\in \cO_{\K}^{m}$.
			\item We can also express the scaling ideal as 
			\begin{equation*}
				\fI_{w}(x)=\prod_{\fp} \fp^{\min_{1\leq i\leq m}\floor*{\frac{\ord_{\fp}x_{i}}{w_{i}}}}.
			\end{equation*}
			for varying prime ideals $\fp\subset\K$.
		\end{enumerate}
	\end{lemma}
	
	\begin{proof}
		\begin{enumerate}
			\item See \cite[Proposition 3.3]{De}.
			\item $\fI_{w}(x)$ is by definition the intersection of all fractional ideals $\fa\subset\K$ satisfying $x_{i}\in \fa^{w_{i}}$ for all $i$. We have 
			\begin{equation*}
				\begin{split}
					x_{i}\in \fa^{w_{i}}&\iff \ord_{\fp}x_{i}\geq w_{i} \ord_{\fp}\fa \,\text{  for all } \fp\\&\iff \ord_{\fp}\fa\leq \floor*{\frac{\ord_{\fp}x_{i}}{w_{i}}}\text{ for all } \fp\\&\iff \fa\supseteq \fp^{\floor*{\frac{\ord_{\fp}x_{i}}{w_{i}}}}\text{ for all } \fp.
				\end{split}
			\end{equation*}
			Thus, $\fI_{w}(x)$ is the intersection of all fractional ideals $\fa\subset\K$ satisfying $\fa\supseteq \bigcap_{\fp}\fp^{\min_{1\leq i\leq m}\floor*{\frac{\ord_{\fp}x_{i}}{w_{i}}}}=\prod_{\fp}\fp^{\min_{1\leq i\leq m}\floor*{\frac{\ord_{\fp}x_{i}}{w_{i}}}}$. Taking the intersection yields the result.
		\end{enumerate}
	\end{proof}
	
	\begin{definition}
		\label{size}
		Let $x\in\K^{m}-\{0\}$. We define the \textit{height of $x$ with respect to $w$} to be 
		\begin{equation*}
			S_{w}(x):= \dfrac{1}{N(\fI_{w}(x))}S_{w,\infty}(x), 
		\end{equation*}
		where we set 
		\begin{equation*}
			S_{w,\infty}(x):= \prod_{v\in \Omega_{\K}}\max_{1\le i\le m}|x_{i}|_{v}^{\frac{1}{w_{i}}}.
		\end{equation*}
	\end{definition}
	
	For some background on height functions, it is worth mentioning other works around heights on different spaces. For general machinery of heights on projective spaces, see \cite[Section 8.5]{Si}. On a more abstract level, \cite[Section 2.2]{ElSaZu} defines heights on a very general type of stacks.
	
	\section{Ingredients}
        \label{sec2}
	
	In this section we will introduce various new notions and study some of their properties, which will be used in the proof of our main result in Section~\ref{sec3}.
	
	In the rest of the paper we fix the following notation. Let $\K$ denote a number field, $\cO_{\K}$ the ring of integers in $\K$, $D_{\K}$ the absolute value of the discriminant, $N:=[\K:\Q]$, $\Cl(\K)$ the group of ideal classes, and $\mu(\K)$ the group of roots of unity. Let $w=(w_{1},\dots,w_{m})$ and $u=(u_{1},\dots,u_{m})$ be two tuples of positive integers with $|w|:=w_{1}+\cdots+w_{m}$ such that $|w|\geq 1$, $\phi: \P(w)_{\K}\to\P(u)_{\K}$ a non-constant morphism between weighted projective spaces and $e$ the integer defined in Lemma~\ref{morphism}. If $\fa$ is a fractional ideal of~$\K$, we write $\fa^{w}:=\fa^{w_{1}}\times\cdots\times\fa^{w_{m}}$. We further keep the notation from the beginning of Section~\ref{ant}.
	
	\subsection{Definitions and notations}
	
	Our first step will be to name several sets that we will often make use of throughout the paper.
	
	\begin{notation}
		$Q_{w}:=\cO_{\K}^{\times}\backslash (\K^{m}-\{0\})$, where $\cO_{\K}^{\times}$ acts with the weight $w$ action.
	\end{notation}
	
	We set $R_{\K}$ to be a set of representatives of $\Cl(\K)$; without loss of generality, we assume every element of $R_{\K}$ is an integral ideal. Since the class group of $\K$ is finite, the set $R_{\K}$ is finite. For the remainder of this paper, $c$ will denote a (varying) element of~$\Cl(\K)$ and $\fa\in R_{\K}$ will be the unique element with $[\fa]=c$.

	\begin{notation}
		\label{qwa}
		We denote $Q_{w}^{\fa}:=\{x\in Q_{w} \mid  \fI_{w}(x)=\fa\}$. 
	\end{notation}

    \begin{remark}
    	\begin{enumerate}
    	    \item $Q_{w}^{\fa}$ is well defined by Lemma~\ref{scalingidealother} (1) and the fact that multiplication of a fractional ideal of a ring $R$ by a unit of $R$ leaves the ideal unchanged.
    		\item $Q_{w}^{\fa}\cong\{x\in\P(w)(\K)\mid [\fI_{w}(x)]=c\}$ by Lemma~\ref{scalingidealother} (1).
    		\item $Q_{w}^{\fa}=\{x\in\cO_{\K}^{\times}\backslash (\fa^{w}-\{0\})\mid \fI_{w}(x)=\fa\}$. This follows from $\fI_{w}(x)=\fa\implies \fI_{w}(x)\subseteq\fa\implies x\in \fa^{w}$, by definition of the scaling ideal.
    	\end{enumerate}
    \end{remark}
	
	\begin{definition}
		\label{discrepancy}
                We define a map
	\begin{align*}
        \delta_{\phi}:\P(w)(\K)&\longrightarrow\{\text{fractional ideals of }\K\}\\
	    x&\longmapsto\fI_{u}(\phi (x))\cdot\fI_{w}(x)^{-e}.
	\end{align*}
        For $x\in\P(w)(\K)$, the fractional ideal $\delta_{\phi}(x)$ is called the \emph{discrepancy ideal} of~$x$ relative to~$\phi$.
	The image of $\delta_{\phi}$ is called the \textit{set of discrepancy ideals of the morphism~$\phi$} and is denoted by $\cD_{\phi}$.
	\end{definition}
	
	\begin{remark}
		 The quotient $\fI_{u}(\phi (x))\cdot\fI_{w}(x)^{-e}$ is well defined for $x\in \P(w)(\K)$, even though the individual factors are not. To see this, note that for $\alpha\in\cO_{\K}^{\times}$ we have $\phi(\alpha x)=\alpha^{e}\phi(x)$, where $\alpha$ acts on $x$ with the weight $w$ action and $\alpha^{e}$ acts on $\phi(x)$ with the weight $u$ action. Then, by Lemma~\ref{scalingidealother} (1), $\fI_{u}(\phi(\alpha x))\cdot\fI_{w}(\alpha x)^{-e}=\alpha^{e}\fI_{u}(\phi(x))\cdot\alpha^{-e}\fI_{w}(x)^{-e}$, and the result follows.
	\end{remark}

	\begin{lemma}
		\label{finiteness}
		If $e=1$ or $w=(1,\ldots,1)$, then the set $\cD_{\phi}$ is finite.
	\end{lemma}

	\begin{proof}
        We view $\K[x_1,\ldots,x_m]$ and $\K[y_1,\ldots,y_m]$ as graded $\K$-algebras with $x_i$ and $y_j$ homogeneous of degrees~$w_i$ and~$u_j$, respectively.  By Lemma~\ref{morphism}, there is a positive integer~$e$ such that the morphism~$\phi$ is represented by polynomials $f_1,\ldots,f_m\in K[x_1,\ldots,x_m]$, with $f_j$ homogeneous of degree~$eu_j$ for $1\le j\le m$.  For $1\le i\le m$, we write $w_i/e=\nu_i/\delta_i$ with $\nu_i$, $\delta_i$ coprime positive integers.

        For $1\le j\le m$, let $\fa_j$ be the fractional ideal generated by the coefficients of~$f_j$.  Then for all $z\in\K^m\setminus\{0\}$, we have
        \[
        \fI_u(\phi(z))\fI_w(z)^{-e}\subseteq \fI_u(\fa_1,\ldots,\fa_m)
        \]
        by the generalization of \cite[Lemma~6.1]{BrNa} to $m$ variables.

        By the arguments in \cite[\S6]{BrNa}, there are integers $\mu_i>0$ and polynomials $h_{i,l}\in\K[y_1,\ldots,y_m]$ for $1\le i\le m$ and $1\le l\le\mu_i$, with $h_{i,\mu_i}\ne0$, such that each $h_{i,l}$ is homogeneous of degree $l\nu_i$ and
        \[
        x_i^{\mu_i \delta_i} + \sum_{l=1}^{\mu_i} h_{i,l}(f_1,\ldots,f_m)x_i^{(\mu_i-l)\delta_i}=0
        \quad\text{in }\K[x_1,\ldots,x_m].
        \]
        For $1\le i\le m$ and $1\le l\le \mu_i$, let $\fc_{i,l}$ be the fractional ideal generated by the coefficients of $h_{i,l}$.  For $1\le i\le m$, we write
        \[
        \fd_i=\fI_{(1,\ldots,\mu_i)}(\fc_{i,1},\ldots,\fc_{i,\mu_i}).
        \]
        Then for all $z\in\K^m\setminus\{0\}$, we have
        \[
        \fI_{(\nu_1,\ldots,\nu_m)}(z_1^{\delta_1},\ldots,z_m^{\delta_m}) \subseteq \fI_{(\nu_1,\ldots,\nu_m)}(\fd_1,\ldots,\fd_m) \fI_u(\phi(z))
        \]
        by the generalization of \cite[Corollary~6.3]{BrNa} to $m$ variables.

        In the case $e=1$, we have $\nu_i=w_i$ and $\delta_i=1$ for $1\le i\le m$, so $\fI_{(\nu_1,\ldots,\nu_m)}(z_1^{\delta_1},\ldots,z_m^{\delta_m}) = \fI_w(z)$.  On the other hand, in the case $w=(1,\ldots,1)$, we have $\nu_i=1$ and $\delta_i=e$ for $1\le i\le n$, so $\fI_{(\nu_1,\ldots,\nu_m)}(z_1^{\delta_1},\ldots,z_m^{\delta_m}) = \fI_{(1,\ldots,1)}(z_1^e,\ldots,z_m^e) = \fI_{(1,\ldots,1)}(z)^e$.  In both cases, we obtain
        \[
        \fI_u(\phi(z))\fI_w(z)^{-e}\supseteq \fI_{(\nu_1,\ldots,\nu_m)}(\fd_1,\ldots,\fd_m)^{-1}.
        \]

        We conclude that every ideal in $\cD_{\phi}$ is contained in $\fI_u(\fa_1,\ldots,\fa_m)$ and contains $\fI_{(\nu_1,\ldots,\nu_m)}(\fd_1,\ldots,\fd_m)^{-1}$.  This implies that $\cD_\phi$ is finite.
	\end{proof}

        \begin{remark}
        The condition ``$e=1$ or $w=(1,\ldots,1)$'' in Lemma~\ref{finiteness} is necessary.  To see this, suppose $e>1$ and (without loss of generality) $w_1>1$.  For $1\le j\le m$, let $c_j$ denote the coefficient of $x_1^{eu_j/w_1}$ in~$f_j$ if $w_1\mid eu_j$, and $c_j=0$ otherwise.  Because $(f_1,\ldots,f_m)$ defines a morphism, not all $c_j$ are zero.  Let $\fp$ be a prime ideal of~$K$ with $\ord_\fp(c_j)\ge0$ for all $j$.  We choose an element $a\in\K^\times$ with $\ord_\fp(a)=1$ and write $z = (a^{w_1-1},0,\ldots,0)\in\K^m$.
        Then we have $\ord_\fp(\fI_w(z))=0$.  On the other hand, for $1\le j\le m$ we have $f_j(z)=c_j (a^{w_1-1})^{eu_j/w_1}$, and hence $\ord_\fp(f_j(z))\ge\ord_\fp(c_j)+u_j$ because $(w_1-1)e/w_1=(1-1/w_1)e\ge(1-1/2)\cdot 2=1$.  This gives $\ord_\fp(\fI_u(\phi(z)))\ge1$, so $\cD_\phi$ contains an ideal divisible by~$\fp$.  Since there are infinitely many $\fp$ as above, we conclude that $\cD_\phi$ is infinite.
        \end{remark}

	We fix an ideal $\fd\in\cD_{\phi}$.
	
	\begin{notation}
		\label{qwad}
		We denote
		\begin{equation*}
			\begin{split}
				Q_{w}^{\fa,\fd}&:=\{x\in Q_{w}^{\fa}\, |\, \fI_{u}(\phi (x))\cdot\fI_{w}(x)^{-e}=\fd\}\\&=\{x\in Q_{w}^{\fa}\mid  \fI_{u}(\phi (x))=\fd\fa^{e}\}.
			\end{split}
		\end{equation*}
	\end{notation}
	
	We now introduce $\fp$-adic versions of some of the above objects.
        First, we define a map
        \[
        \fI_{w,\fp}:\K_{\fp}^m-\{0\}\longrightarrow\{\text{fractional ideals of }\cO_{\K,\fp}\}
        \]
        in the same way as in Definition~\ref{scalingideal}.  Equivalently, we have
        \[
        \fI_{w,\fp}(x) = (\fp\cO_{\K,\fp})^{\epsilon_{w,\fp}(x)},
        \]
        where $\epsilon_{w,\fp}:\K_{\fp}^m-\{0\}\to\Z$ is defined by
        \[
        \epsilon_{w,\fp}(x) = \min_{1\le i\le m}\floor*{\frac{\ord_{\fp} x_i}{w_i}}.
        \]
        It is straightforward to check that $\epsilon_{w,\fp}$ is continuous.
	
	\begin{notation}
		We denote $Q_{w,\fp}:=\cO_{\K,\fp}^{\times}\backslash(\K_{\fp}^{m}-\{0\})$, where again $\cO_{\K,\fp}^{\times}$ acts with the weight~$w$ action.
	\end{notation}
	
	Let us decompose $\fd=\prod_{\fp}\fp^{j_{\fp}}$ and $\fa=\prod_{\fp}\fp^{k_{\fp}}$. It then follows that $Q_{w}^{\fa}=\{x\in Q_{w}\mid\text{for all }  \fp: \fI_{w,\fp}(x)=\fa_{\fp}\}$, where $\fa_{\fp}=\fa\cO_{\K,\fp}=(\fp\cO_{\K,\fp})^{k_{\fp}}$.
	
	\begin{notation}
		 We denote $R_{\fp}^{\fa}:=\left\{x\in \K_{\fp}^{m}-\{0\}\mid \epsilon_{w,\fp}(x) = k_{\fp}\right\}$.
	\end{notation}
	
	\begin{lemma}
		\label{primitive}
                \begin{enumerate}
		\item We have $R_{\fp}^{\fa}=(\fa_{\fp})^{w}-(\fp\fa_{\fp})^{w}$.
		\item $R_{\fp}^{\fa}$ is compact.
                \end{enumerate}
	\end{lemma}
	
	\begin{proof}
		This is a straightforward verification.
	\end{proof}
	
	\begin{notation}
		We denote $Q_{\fp}^{\fa}:=\cO_{\K,\fp}^{\times}\backslash R_{\fp}^{\fa}$. 
	\end{notation}
	
	We then consider the following diagram, where the maps are the canonical ones.
	\begin{equation*}
		\xymatrix{Q_{w} \ar[r]^{\pi_{\fp}} & Q_{w,\fp}\\
			Q_{w}^{\fa}\ar@{^(->}[u] \ar[r] & Q_{\fp}^{\fa}\ar@{^(->}[u]^{q}}
	\end{equation*}
	By Lemma~\ref{primitive}, $R_{\fp}^{\fa}$ is simply the subset of $\K_{\fp}^{m}$ defined by the equality $\fI_{w,\fp}(x)=\fa_{\fp}$, the local analogue of the requirement that $\fI_{w}(x)=\fa$. It is then clear that $Q_{w}^{\fa}=\bigcap_{\fp}\pi_{\fp}^{-1}Q_{\fp}^{\fa}$.
	
	We now derive a similar expression for $Q_{w}^{\fa,\fd}$. To do so, we first define a map
        \[
        \delta_{\phi,\fp}:\P(w)(\K_p)\longrightarrow\Z
        \]
        by
        \begin{align*}
        \delta_{\phi,\fp}(x) &:= \ord_{\fp}(\fI_{u,\fp}(\phi(x))\cdot\fI_{w,\fp}(x)^{-e})\\
        &= \epsilon_{u,\fp}(\phi(x)) - e\epsilon_{w,\fp}(x).
        \end{align*}
        This is a local analogue of the map $\delta_{\phi}$.  It is straightforward to check that $\delta_{\phi,\fp}$ is well defined and continuous.
        We then use the following diagram, in which $[\cdot]$ is the canonical map:
	\begin{equation*}
		\xymatrix{R_{\fp}^{\fa}\ar@{^(->}[r]\ar@{->>}[d] & \K_{\fp}^{m}-\{0\}\ar@{->>}[d]^{[\cdot]}\\
			Q_{\fp}^{\fa} \ar[r]& \P(w)(\K_{\fp})\ar[d]^{\delta_{\phi,\fp}}\\  & \Z \\}
	\end{equation*}
	
	\begin{notation}
		\label{rpkpjp}
		We denote $R_{\fp}^{\fa,\fd}:=\{x\in R_{\fp}^{\fa}\, |\, \delta_{\phi,\fp}([x])=j_{\fp}\}$.
	\end{notation}
	
	\begin{notation}
		We denote $Q_{\fp}^{\fa,\fd}:=\cO_{\K,\fp}^{\times}\backslash R_{\fp}^{\fa,\fd}$.
	\end{notation}
	
	We conclude that $Q_{w}^{\fa,\fd}=\bigcap_{\fp}\pi_{\fp}^{-1} Q_{\fp}^{\fa,\fd}$ by the same argument as above, combining the requirements that the scaling ideal should be $\fa$ and the discrepancy ideal should be $\fd$.

	\subsection{Moduli and periodic sets}
	
	Our next step will be to introduce some moduli together with the remaining few relevant definitions, which we will utilise to ultimately study our main counting issue.
	
	By construction, working with $Q_{\fp}^{\fa,\fd}$ (resp.\ $Q_{\fp}^{\fa}$) is equivalent to working with subsets of $R_{\fp}^{\fa,\fd}$ (resp.\ $R_{\fp}^{\fa}$) that are stable under multiplication by $\cO_{\K,\fp}^{\times}$. For simplicity we will only refer to $R_{\fp}^{\fa,\fd}$ (resp.\ $R_{\fp}^{\fa}$), but note that analougous results hold for $Q_{\fp}^{\fa,\fd}$ (resp.\ $Q_{\fp}^{\fa}$).
	
	\begin{remark}
		\label{rpkpjpproperties}
		It follows from the continuity of $\delta_{\phi,\fp}$ that $R_{\fp}^{\fa,\fd}$ is open and closed in $R_{\fp}^{\fa}$, and thus compact by Lemma~\ref{primitive}.
	\end{remark}
	
	As a consequence, we get the following lemma.
	
	\begin{lemma}
		\label{congruenceconditions}
		For every prime $\fp$ there exists an open subset of the form $\fm_{\fp}^{\fa,\fd}=\fm_{\fp,1}^{\fa,\fd}\times\cdots\times \fm_{\fp,m}^{\fa,\fd}\subseteq (\fa_{\fp})^{w}$, such that $R_{\fp}^{\fa,\fd}$ is determined by congruence conditions modulo $\fm_{\fp}^{\fa}:=\bigcap_{\fd}\fm_{\fp}^{\fa,\fd}$ in $R_{\fp}^{\fa}$. Moreover, for every $\fp$ with $\ord_{\fp}\fd=0$ for all $\fd\in\cD_{\phi}$, we can choose $\fm_{\fp}^{\fa}=(\fa_{\fp})^{w}$.
	\end{lemma}
	
	\begin{proof}
		We view $R_{\fp}^{\fa,\fd}$ as a subset of $(\fa_{\fp})^{w}$. The products
		\begin{equation*}
			\cU_{\fp}^{a^{\fp},b^{\fp}}:=\prod_{i=1}^{m}(a_{i}^{\fp}+\fp^{b_{i}^{\fp}}\fa_{\fp}^{w_{i}})
		\end{equation*}
		for varying $a^{\fp}=(a_{1}^{\fp},\dots,a_{m}^{\fp})\in\fa^{w}$ and $b^{\fp}=(b_{1}^{\fp},\dots,b_{m}^{\fp})\in\Z_{\geq 0}^{m}$ define a basis of open subsets for $(\fa_{\fp})^{w}$. For all but finitely many $\fp$, let us choose $\cU_{\fp}^{a^{\fp},b^{\fp}}$ to be the whole $(\fa_{\fp})^{w}$. That way, $\prod_{\fp}\cU_{\fp}^{a^{\fp},b^{\fp}}$ is also a basis for $\prod_{\fp} (\fa_{\fp})^{w}$. Without loss of generality, we may choose $\cU_{\fp}^{a^{\fp},b^{\fp}}=(\fa_{\fp})^{w}$ for every $\fp$ such that $\ord_{\fp}\fd=0$ for all $\fd\in\cD_{\phi}$.
		
		By Remark~\ref{rpkpjpproperties}, we can cover $R_{\fp}^{\fa,\fd}$ by a finite collection of the aforementioned basic open subsets intersected with $R_{\fp}^{\fa,\fd}$. That is, we can write
		\begin{equation*}
			R_{\fp}^{\fa,\fd}=\bigcup_{j\in\cJ}(\cU_{\fp}^{a^{\fp,j},b^{\fp,j}}\cap R_{\fp}^{\fa,\fd})
		\end{equation*} 
		where $\cJ$ is a finite index set.  Then there is a maximal value of $b_{i}^{\fp,j}$ as $j$ varies and $i$ is fixed. We denote this value by $s_{i}:=\max_{j\in\cJ}(b_{i}^{\fp,j})$, and $s:=(s_{1},\dots,s_{m})$. Note that for $\fp$ with $\ord_{\fp}\fd=0$ for all $\fd\in\cD_{\phi}$ we have $b_{1}^{\fp,j}=\cdots=b_{m}^{\fp,j}=0$ and thus $s_{i}=0$ for all $i$.
		
		By abuse of notation, let us denote $\cU_{\fp}^{a^{\fp,j},b^{\fp,j}}:=\cU_{\fp}^{a^{\fp,j},b^{\fp,j}}\cap R_{\fp}^{\fa,\fd}$. Each $a_{i}^{\fp,j}+\fp^{b_{i}^{\fp,j}}\fa_{\fp}^{w_{i}}$ is stable under translation by $\fp^{s_{i}}\fa_{\fp}^{w_{i}}$ and we conclude that $\cU_{\fp}^{a^{\fp,j},b^{\fp,j}}$ is $(\fp^{s_{1}}\fa_{\fp}^{w_{1}}\times\cdots\times \fp^{s_{m}}\fa_{\fp}^{w_{m}})$-periodic. We can therefore consider the following diagram:
		\begin{equation*}
			\xymatrix{\cU_{\fp}^{a^{\fp,j},b^{\fp,j}}\ar@{->>}[r] & \cU_{\fp}^{a^{\fp,j},b^{\fp,j}}/ (\fp^{s_{1}}\fa_{\fp}^{w_{1}}\times\cdots\times \fp^{s_{m}}\fa_{\fp}^{w_{m}}) \ar@{^(->}[d]\\(\fa_{\fp})^{w}\ar@{->>}[r]^(.30){q}&(\fa_{\fp})^{w} / (\fp^{s_{1}}\fa_{\fp}^{w_{1}}\times\cdots\times \fp^{s_{m}}\fa_{\fp}^{w_{m}})}
		\end{equation*}
		
		We can easily deduce from the diagram that $\cU_{\fp}^{a^{\fp,j},b^{\fp,j}}$ is the inverse image of  $\cU_{\fp}^{a^{\fp,j},b^{\fp,j}}/ (\fp^{s_{1}}\fa_{\fp}^{w_{1}}\times\cdots\times \fp^{s_{m}}\fa_{\fp}^{w_{m}})$ under the map $q$, implying that $\cU_{\fp}^{a^{\fp,j},b^{\fp,j}}$ is determined by congruence conditions modulo $\fp^{s_{1}}\fa_{\fp}^{w_{1}}\times\cdots\times \fp^{s_{m}}\fa_{\fp}^{w_{m}}$. The result follows.
	\end{proof}
	
	\begin{definition}
		We call $\fm_{\fp}^{\fa}$ the \textit{local modulus at $\fp$}.
	\end{definition}
	
	\begin{remark}
		Let $\fp$ be such that $k_{\fp}=0$ and $\ord_{\fp}\fd=0$ for all $\fd\in\cD_{\phi}$. Then $\fm_{\fp}^{\fa}=\cO_{\K,\fp}^{m}$, since for all the primes $\fp$ that do not divide $\fa$ the ideal $\fa_{\fp}=(\fp\cO_{\K,\fp})^{k_{\fp}}$ is trivial.
	\end{remark}
	
	\begin{remark}
		\label{maindiagram}
		The first statement of Lemma~\ref{congruenceconditions} can be translated as follows: there exists a subset $S_{\fp}^{\fa,\fd}$ in $(\fa_{\fp}^{w_{1}}/\fm_{\fp,1}^{\fa})\times\cdots\times(\fa_{\fp}^{w_{m}}/\fm_{\fp,m}^{\fa})$ such that $R_{\fp}^{\fa,\fd}$ is the inverse image of $S_{\fp}^{\fa,\fd}$ under the reduction map shown in the diagram below. To see this, note that taking the inverse image under the reduction map is essentially the same as taking the points of $R_{\fp}^{\fa,\fd}$ satisfying some congruence conditions modulo $\fm_{\fp}^{\fa}$.
		\begin{equation*}
			\xymatrix{R_{\fp}^{\fa}\ar@{^(->}[r]  & (\fa_{\fp})^{w}\ar@{->>}[r]  & (\fa_{\fp}^{w_{1}}/\fm_{\fp,1}^{\fa})\times\cdots\times(\fa_{\fp}^{w_{m}}/\fm_{\fp,m}^{\fa})\\
				R_{\fp}^{\fa,\fd}\ar@{^(->}[u]\ar[rr]  &  &  S_{\fp}^{\fa,\fd}\ar@{^(->}[u]}
		\end{equation*}
	\end{remark}
	
	All sets in the above diagram, except possibly $S_{\fp}^{\fa,\fd}$, are stable under the (weight~$w$) action of $\cO_{\K,\fp}^{\times}$, since fractional ideals are stable under multiplication by units. Without loss of generality we can choose $S_{\fp}^{\fa,\fd}$ to be $\cO_{\K,\fp}^{\times}$-stable as well, because we can replace $S_{\fp}^{\fa,\fd}$ by $\bigcup_{\kappa\in\cO_{\K,\fp}^{\times}}\kappa S_{\fp}^{\fa,\fd}$, making it $\cO_{\K,\fp}^{\times}$-stable while maintaining $R_{\fp}^{\fa,\fd}$ as its inverse image. For the remainder of the paper we fix such a choice of $\cO_{\K,\fp}^{\times}$-stable sets $S_{\fp}^{\fa,\fd}$.
	
	\begin{definition}
		\label{mad}
		We define the \textit{global modulus} to be
		\begin{equation*}
			\begin{split}
				\fm^{\fa}&:=\bigcap_{\fp}(\fm_{\fp}^{\fa}\cap \K^{m})=\fm_{1}^{\fa}\times\cdots\times\fm_{m}^{\fa},
			\end{split}
		\end{equation*}
		where $\fm_{i}^{\fa}:=\bigcap_{\fp}(\fm_{\fp,i}^{\fa}\cap \K)$ and the intersection inside the brackets takes place in $\K_{\fp}$. Equivalently, $\fm^{\fa}=\{x\in \K^{m}\, |\, \text{ for all } \fp: x\in \fm_{\fp}^{\fa}\}$. Note that $\fm^{\fa}$ is a lattice in $\K^{m}$.
	\end{definition}
	
	\begin{remark}
		\label{mpadokp}
		It follows from the second result of Lemma~\ref{congruenceconditions} that this set is well defined. Else, we would have non-trivial conditions at every $\fp$.  
	\end{remark}
	
	\begin{notation}
		\label{vpkpjp}
		We denote by $V_{\fp}^{\fa,\fd}$ the inverse image of $S_{\fp}^{\fa,\fd}$ in $(\fa_{\fp})^{w}$.
	\end{notation}
	
	\begin{remark}
		\label{vpkpjprpkp}
		Note that $R_{\fp}^{\fa,\fd} =V_{\fp}^{\fa,\fd}\cap R_{\fp}^{\fa}$.
	\end{remark}
	
	\begin{remark}
		\label{rpkpjprpkp}
		Let $\fp$ such that $\fm_{\fp}^{\fa}=(\fa_{\fp})^{w}$ is trivial. Then, $S_{\fp}^{\fa,\fd}$ is also trivial. As a consequence, we get that $V_{\fp}^{\fa,\fd}=(\fa_{\fp})^{w}$ and $R_{\fp}^{\fa,\fd}=R_{\fp}^{\fa}$.
	\end{remark}
	
	We will next introduce one of our most important concepts, which we will use in Section~\ref{sec3} for counting points within a suitable weighted expanding compact set, a key step towards our main result.
	
	\begin{definition}
		\label{vad}
		We define the \textit{global set} $V^{\fa,\fd}$ to be  
		\begin{equation*}
			V^{\fa,\fd}:=\bigcap_{\fp}\pi_{\fp}^{-1}(V_{\fp}^{\fa,\fd}).
		\end{equation*}
		In other words, $V^{\fa,\fd}$ is the set of points in $\fa^{w}$ that, localized, satisfy all the congruence conditions modulo $\fm_{\fp}^{\fa}$ at every $\fp$.
	\end{definition}
	
	\begin{example}
		\label{phiidentity}
		Let $\phi$ be the identity map. Then $e=1$ and there is only one choice for $\fd\in\cD_{\phi}$, namely $\fd=(1)$, and thus by Lemma~\ref{congruenceconditions} we can take $\fm_{\fp}^{\fa}$ to be $(\fa_{\fp})^{w}$ for all $\fp$, which by Remark~\ref{rpkpjprpkp} implies that for all $\fp$ we have $V_{\fp}^{\fa,\fd}=(\fa_{\fp})^{w}$ and therefore $V^{\fa,\fd}=\fa^{w}$.
	\end{example}
	
	\begin{lemma}
		\label{vadproperties}
		$V^{\fa,\fd}$ satisfies the following properties:
		\begin{enumerate}
			\item $V^{\fa,\fd}$ is $\fm^{\fa}$-periodic. As a result, $V^{\fa,\fd}$ is a union of cosets of $\fm^{\fa}$ and we can decompose $V^{\fa,\fd}$ as $V^{\fa,\fd}=\bigsqcup_{x\in X}(x+\fm^{\fa})$ for some finite subset $X\subset V^{\fa,\fd}$.
			\item $V^{\fa,\fd}$ is stable under the (weight $w$) action of $\cO_{\K}^{\times}$ on $\K^{m}$.
		\end{enumerate}
	\end{lemma}
	
	\begin{proof}
		Both results follow from analogous properties of $V_{\fp}^{\fa,\fd}$: 
		\begin{enumerate}
			\item Let $x\in V_{\fp}^{\fa,\fd}$ and $y\in \fm_{\fp}^{\fa}$. Then it is clear by definition that $x+y\in V_{\fp}^{\fa,\fd}$ and so $V_{\fp}^{\fa,\fd}$ is $\fm_{\fp}^{\fa}$-periodic for each $\fp$. Thus the result follows.
			\item By definition of $V_{\fp}^{\fa,\fd}$ and the fact that $S_{\fp}^{\fa,\fd}$ is $\cO_{\K,\fp}^{\times}$-stable together with the compatibility of the reduction map, we ensure that $V_{\fp}^{\fa,\fd}$ is $\cO_{\K,\fp}^{\times}$-stable as well. Again, the construction of $V^{\fa,\fd}$ makes the result straightforward, as $\cO_{\K}^{\times}\subseteq\cO_{\K,\fp}^{\times}$ for any $\fp$.
		\end{enumerate}
	\end{proof}
	
	\begin{remark}
		\label{remark}
		Note that $Q_{w}^{\fa,\fd}=\{x\in \cO_{\K}^{\times}\backslash (V^{\fa,\fd}-\{0\})\mid \fI_{w}(x)=\fa\}$. This follows from the expression $Q_{w}^{\fa,\fd}=\bigcap_{\fp}\pi_{\fp}^{-1}(\cO_{\K,\fp}^{\times}\backslash R_{\fp}^{\fa,\fd})$ and the fact that $\{x\in V^{\fa,\fd}-\{0\}\mid \fI_{w}(x)=\fa\}=\bigcap_{\fp}\pi_{\fp}^{-1}(V_{\fp}^{\fa,\fd}\cap R_{\fp}^{\fa})=\bigcap_{\fp}\pi_{\fp}^{-1}(R_{\fp}^{\fa,\fd})$.
	\end{remark}
	
	\section{Counting $\K$-points of $\P(w)$}
        \label{sec3}
	
	In Theorem~\ref{maintheorem}, we will find an asymptotic formula for the number of $\K$-rational points of the weighted projective space $\P(w)$ whose image under~$\phi$ has bounded height, as in Definition~\ref{size}. Due to the `stacky' nature of $\P(w)$, it is convenient to count each $x\in\P(w)(\K)$ with weight $1/\#{\Aut x}$, where
	\[
	\Aut x = \{\lambda\in\K^\times\mid \lambda_* x = x \}.
	\]
	More generally, let $X$ be a finite set and let $\Aut x$ be an attached finite group for each $x\in X$. Then we write
	\[
	\widehat{\#}X = \sum_{x\in X}\frac{1}{\#{\Aut x}}.
	\]
	(This is sometimes called the \emph{mass} or \emph{groupoid cardinality} of~$X$.)
	Our final aim will therefore be to determine the asymptotic behaviour, as $T\to\infty$, of the quantity
	\begin{equation*}
		N_{\phi}(T):=\widehat{\#}\{x\in\P(w)(\K) \mid S_{u}(\phi(x))\leq T\}.
	\end{equation*} 
	We do so by following a method analogous to that of \cite{De}. 

	\subsection{Introduction to the counting problem}
	
	We will first introduce some values similar to $N_{\phi}(T)$, through which we intend to ease the computation of $N_{\phi}(T)$. We first denote
	\begin{equation*}
		N_{\phi}(c,T):=\widehat{\#}\{x\in\P(w)(\K) \mid  [\fI_{w}(x)]=c, \, S_{u}(\phi(x))\leq T\}.
	\end{equation*}
    We achieve our counting goal by summing $N_{\phi}(c,T)$ over all the ideal classes $c$ of $\K$.
	
	\begin{remark}
		\label{nctequal}
		Note that 
		\begin{equation*}
			\begin{split}
				N_{\phi}(c,T)&=\widehat{\#}\{x\in Q_{w}^{\fa} \mid  S_{u}(\phi(x))=\frac{S_{u,\infty}(\phi(x))}{N(\fI_{u}(\phi(x)))}\leq T\}\\
				&=\sum_{\fd\in\cD_{\phi}}\widehat{\#}\{x\in Q_{w}^{\fa}\mid  \fI_{u}(\phi (x))=\fd\fa^{e}, S_{u,\infty}(\phi(x))\leq N(\fd\fa^{e})T\}\\
				&=\sum_{\fd\in\cD_{\phi}}\widehat{\#}\{x\in Q_{w}^{\fa,\fd}\mid S_{u,\infty}(\phi(x))\leq N(\fd\fa^{e})T\}.
			\end{split}
		\end{equation*}
	\end{remark}
	
	Let $\fb\subseteq\fa$ be an ideal. We define the following two quantities. 
	\begin{align*}
		\bar{\cN}_{\phi}(\fb,\fa,\fd,T)&:=\widehat{\#}\{x\in \cO_{\K}^{\times}\backslash (V^{\fa,\fd}-\{0\}) \mid  \fI_{w}(x)=\fb,S_{u,\infty}(\phi(x))\leq T\},\\
		\cN_{\phi}(\fb,\fa,\fd,T)&:=\widehat{\#}\{x\in \cO_{\K}^{\times}\backslash (V^{\fa,\fd}-\{0\}) \mid  \fI_{w}(x)\subseteq\fb,\, S_{u,\infty}(\phi(x))\leq T\}.
	\end{align*}
	By Lemma~\ref{vadproperties} (2), these are well defined, and it is clear that the following equality holds:
	\begin{equation*}
		\cN_{\phi}(\fb,\fa,\fd,T)=\sum_{\fe\subseteq \cO_{\K}}\bar{\cN}_{\phi}(\fb\fe,\fa,\fd,T).
	\end{equation*}
    We can now apply Möbius inversion, from which we deduce that 
	\begin{equation*}
		\bar{\cN}_{\phi}(\fb,\fa,\fd,T)=\sum_{\fc\subseteq \cO_{\K}}\mu_{\K}(\fc)\cN_{\phi}(\fb\fc,\fa,\fd,T).
	\end{equation*}
	In this way, we can reduce the counting of points with fixed scaling ideal (which is what we ultimately need to compute) to simply counting points with scaling ideal contained in a fixed ideal.
	
    On top of that, it is not hard to see, following Remark~\ref{nctequal} and Remark~\ref{remark}, that one has
	\begin{equation*}
		N_{\phi}(c,T)=\sum_{\fd\in\cD_{\phi}}\bar{\cN}_{\phi}(\fa,\fa,\fd,TN(\fd)N(\fa)^{e}).
	\end{equation*}
	Hence, we gather that 
	\begin{equation}
		\label{equation}
		\begin{split}
			N_{\phi}(T)&=\sum_{c\in \Cl(\K)}N_{\phi}(c,T)\\&=\sum_{\fa\in R_{\K}}\sum_{\fd\in\cD_{\phi}}\bar{\cN}_{\phi}(\fa,\fa,\fd,TN(\fd)N(\fa)^{e})\\&=\sum_{\fa\in R_{\K}}\sum_{\fd\in\cD_{\phi}}\sum_{\fc\subseteq \cO_{\K}}\mu_{\K}(\fc)\cN_{\phi}(\fa\fc,\fa,\fd,TN(\fd)N(\fa)^{e}).
		\end{split}
	\end{equation}
	This reduces the original problem to determining the asymptotic behaviour of 
	\begin{equation*}
		\cN_{\phi}(\fa\fc,\fa,\fd,T) \mbox{ as } T\to\infty
	\end{equation*}
	for
	\begin{enumerate}
		\item finitely many fixed representatives $\fa$ of $\Cl(\K)$,
		\item finitely many $\fd\in \cD_{\phi}$,
		\item all integral ideals $\fc\subseteq \cO_{\K}$.
	\end{enumerate}
	
	\subsection{Reduction to a lattice point problem}
        \label{lattice-point-problem}
	
	Mirroring the method used in \cite[Section 4]{De}, we may simplify the problem even further by making it into a lattice point counting issue. 
	
	By Dirichlet's unit theorem, the image of the group homomorphism 
		\begin{align*}
			\ell: \cO_{\K}^{\times} &\longrightarrow \R^{\Omega_{\K}}\\
			u &\longmapsto (\log|u|_{v})_{v\in \Omega_{\K}}
		\end{align*}
	is a lattice $\Gamma$ of rank $r:=r_{1}+r_{2}-1$ in the hyperplane $\cH$ defined by $\sum _{v\in \Omega_{\K}}y_{v}=0$. Moreover, $\ker(\ell)=\mu(\K)$.
    As in \cite[Section 4]{De}, we define maps
    \begin{align*}
    \pr:\R^{\Omega_{\K}}&\longrightarrow \cH\\
    y&\longmapsto \left(y_{v}-\left(\frac{\sum_{v'\in \Omega_{\K}}y_{v'}}{N}\right)N_{v}\right)_{v\in\Omega_{\K}}
    \end{align*}
    and
    \begin{align*}
    \eta: \prod_{v\in \Omega_{\K}} (\K_{v}^{m}-\{0\}) &\longrightarrow \R^{\Omega_{\K}}\\
		(z_v)_{v\in\Omega_{\K}} &\longmapsto \left(\log\max_{1\le i\le m}|z_{v,i}|_{v}^{\frac{1}{w_{i}}}\right)_{v\in\Omega_{\K}}.
	\end{align*}
	Let $F$ be the standard fundamental domain for $\cH/\Gamma$ determined by a $\Z$-basis of~$\Gamma$. Then we define
	\[
	\Delta := ({\pr}\circ \eta)^{-1}F \subseteq \prod_{v\in\Omega_{\K}}(\K_v^m-\{0\}).
	\]
	
	\begin{lemma}[{\cite[Proposition 4.1, Lemma 4.3]{De}}]
	\label{del}
	The set $\Delta$ has the following properties:
		\begin{enumerate}
		    \item $\Delta$ is $\mu(\K)$-stable and there is a bijection
	    \[
	    \mu(\K)\backslash\Delta \isom \cO_\K^\times\backslash \prod_{v\in\Omega_{\K}}(\K_v^m-\{0\}).
	    \]
			\item $t_{*}\Delta=\Delta$ for all $t\in\R^{\times}$.
		\end{enumerate}
	\end{lemma}
	
	\begin{proof}
	    For the first claim, note that $(\cO_\K^\times)_{*}\Delta = \prod_{v\in\Omega_{\K}}(\K_v^m-\{0\})$ and for all $x\in\Delta$ and $a\in\cO_\K^\times$ we have $a_{*}x\in\Delta$ if and only if $a\in\mu(\K)$. The second claim easily follows from the definition of $\Delta$.
	\end{proof}
	
    \begin{corollary}
	\label{orbits}
	    Let $Z$ be an $\cO_\K^\times$-stable subset of $\prod_{v\in\Omega_{\K}}(\K_v^m-\{0\})$ consisting of finitely many orbits. To each orbit $(\cO_\K^\times)_{*}x\in \cO_\K^\times\backslash Z$ we associate the automorphism group $\Aut x=\{a\in\cO_\K^\times\mid a_{*}x=x\}$. Then we have
	    \[
	    \widehat\#(\cO_\K^\times\backslash Z) = \#\mu(\K)^{-1}\#(Z\cap\Delta).
	    \]
	\end{corollary}
	
	\begin{proof}
	    The bijection in Lemma~\ref{del} (1) restricts to a bijection
	    \[
	    \mu(\K)\backslash(Z\cap \Delta) \isom \cO_\K^\times\backslash Z.
	    \]
	    This respects the automorphism groups as these are all contained in $\mu(\K)$.
	    This implies
	    \[
	    \widehat\#(\cO_\K^\times\backslash Z) = \widehat\#(\mu(\K)\backslash(Z\cap \Delta)).
	    \]
	    By the orbit-stabiliser theorem, the right-hand side equals $\#\mu(\K)^{-1}\#(Z\cap\Delta)$.
	\end{proof}
	
	We now introduce analogues of the sets $\fD(T)$ and $\fB(T)$ in \cite[Section~4]{De}.
	
	\begin{notation}
		\label{dt}
		For all $T>0$, we write
		\begin{equation*}
			\fD_{\phi}(T):=\biggl\{z\in\prod_{v\in \Omega_{\K}}(\K_{v}^{m}-\{0\})\biggm|
			\prod_{v\in \Omega_{\K}}\max_{1\le i\le m}|\phi(z)_{i}|_{v}^{1/u_{i}}
			\leq T\biggr\}
		\end{equation*}
		and $\fB_{\phi}(T):=\fD_{\phi}(T)\cap\Delta$.
	\end{notation}

	\begin{remark}
		The set $\fD_{\phi}(T)$ is $\cO_{\K}^{\times}$-stable, because for all $\lambda\in\cO_{\K}^{\times}$ and $v\in\Omega_{\K}$ we have
		\[
			\max_{i}|\phi(\lambda z)_{i}|_{v}^{1/u_{i}}=\max_{i}|\lambda^{eu_{i}}\phi(z)_{i}|_{v}^{1/u_{i}}=|\lambda^{e}|_{v}\max_{i}|\phi(z)_{i}|_{v}^{1/u_{i}}
		\]
		and $\prod_{v\in\Omega_{\K}}|\lambda^e|_v=1$.
	\end{remark}
	
	\begin{lemma}
		\label{bt}
		We have $\fB_{\phi}(T)=T_{*}^{1/eN}\fB_{\phi}(1)$ with the weight $w$ action.
	\end{lemma}
	
    \begin{proof}
	This follows by an argument similar to that in \cite[Lemma~4.3]{De}.
	\end{proof}
	
	Next, we prove an analogue of \cite[Proposition 4.2]{De} in our setting, which will let us switch from counting $\cO_{\K}^{\times}$-orbits to counting lattice points instead.
		
    \begin{proposition}
        \label{4.2}
        We have
        \[
        \cN_{\phi}(\fb,\fa,\fd,T) = \#\mu(\K)^{-1}\#\bigl(V^{\fa,\fd}\cap\fb^w\cap\fB_\phi(T)\bigr).
        \]
    \end{proposition}

    \begin{proof}
        For $x\in\cO_\K^\times\backslash(V^{\fa,\fd}-\{0\}),$ the conditions $\fI_w(x)\subseteq\fb$ and $S_{u,\infty}(\phi(x))\le T$ in the definition of $\cN_\phi(\fb,\fa,\fd,T)$ are equivalent to $x\in\fb^w$ and $x\in\fD_\phi(T)$, respectively.
	    Since moreover $0$ is not in $\fD_\phi(T)$, we can write
    	\[
	    \cN_\phi(\fb,\fa,\fd,T) = \widehat{\#}\bigl(\cO_\K^\times\backslash \bigl(V^{\fa,\fd}\cap\fb^w \cap \fD_\phi(T)\bigr)\bigr).
	    \]
		By Corollary~\ref{orbits}, this can be rewritten as
        \[
        \cN_{\phi}(\fb,\fa,\fd,T) = \#\mu(\K)^{-1}\#\bigl(V^{\fa,\fd}\cap\fb^w\cap\fD_\phi(T)\cap\Delta\bigr).
        \]
        The claim now follows from the definition of $\fB_\phi(T)$.
    \end{proof}

	\subsection{The asymptotic formula for $\cN_{\phi}(\fb,\fa,\fd,T)$}
	
	The next step towards our objective is to generalize \cite[Proposition 4.4]{De}, which will allow us to count lattice points.
        We will use some theory on o-minimal structures; we refer to Barroero and Widmer \cite{BW} for the definitions and results that we use.

        \begin{lemma}
	\label{definable}
	    The subsets $\Delta$, $\fD_\phi(T)$ and $\fB_\phi(T)$ of $\prod_{v\in\Omega_{\K}}\K_v^m$ are definable in the o-minimal structure $\R_{\exp}$, and $\fB_\phi(T)$ is bounded.
	\end{lemma}
	
	\begin{proof}
	The set $\Delta$ is defined by inequalities involving logarithms and absolute values and is therefore definable in $\R_{\exp}$.  The set $\fD_\phi(T)$ is semi-algebraic and hence definable as well.  It follows that the intersection $\fB_\phi(T)$ is definable in $\R_{\exp}$.

        For all $v\in\Omega_{\K}$ the continuous function
        \begin{align*}
        \P(w)(\K_v)&\longrightarrow\R_{>0}\\
        z&\longmapsto\frac{\max\limits_{1\le j\le m}|f_j(z)|_v^{1/u_j}}
        {\max\limits_{1\le i\le m}|z_i|_v^{e/w_i}}
        \end{align*}
        is bounded by compactness of $\P(w)(\K_v)$.  This implies that there exists $C_\phi>0$ such that for all $z\in\fD_\phi(T)$ we have
        \[
        \sum_{v\in\Omega_{\K}}\log\max_{1\le i\le m}|z_{v,i}|_v^{1/w_i}\le\log(C_\phi T^{1/e}).
        \]
	If in addition we have $z\in\Delta$, then
        $\bigl(\log\max_{1\le i\le m}|z_{v,i}|_v^{1/w_i}\bigr)_{v\in\Omega_{\K}}$ lies in $\pr^{-1}F$.
        It follows that if $z$ is in $\fB_\phi(T)$, then $\max_{1\le i\le m}|z_{v,i}|_v$ is bounded for all $v\in\Omega_{\K}$, so $\fB_\phi(T)$ is bounded.
        \end{proof}
	
	Let $\fP$ a subset of $\R^{k}$, and let $(w_1,\ldots,w_k)$ be a tuple of positive integers.  We consider the \textit{weighted expanding set} $T_*\fP$ defined in \cite[Section 4]{De}, i.e.
	\begin{equation*}
		T_*\fP=\{(T^{w_{1}}x_{1},\dots,T^{w_{k}}x_{k}):(x_{1},\dots,x_{k})\in\fP\}.
	\end{equation*}

	We may now state a generalization of \cite[Proposition 4.4]{De} to lattice cosets.
	
	\begin{lemma} 
		\label{4.4}
		Let $\fP$ be a bounded subset of $\R^k$ definable in $\R_{\exp}$, let $\Lambda$ be a lattice in $\R^{k}$ and let $x\in\R^{k}$. We have
		\begin{equation*}
			\#((x+\Lambda)\cap T_*\fP)=\frac{\vol_{k}(\fP)}{\det\Lambda}T^{|w|}+O(T^{|w|-\min\{w_1,\ldots,w_k\}})
                        \quad\text{as }T\to\infty,
		\end{equation*}
                with an implied constant depending on $\fP$ and~$\Lambda$.
	\end{lemma}
	
	\begin{proof}
        We consider the definable set
        \[
        Z=\bigl\{(T,x,z)\in\R\times\R^k\times\R^k\bigm|
        T>0\text{ and }x+z\in T_*\fP\bigr\},
        \]
        viewed as a family of definable subsets of $\R^k$ parametrized by $(T,x)$.  Then all the fibres
        \[
        Z_{(T,x)} = \{z\in\R^k\mid x+z\in T_*\fP\}
        \]
        are bounded, and we have
        \[
        \#(Z_{(T,x)}\cap\Lambda) = \#((x+\Lambda)\cap T_*\fP).
        \]
        The claim now follows by a result of Barroero and Widmer \cite[Theorem~1.3]{BW}.
	\end{proof}
	
	Recall from Lemma~\ref{vadproperties} that we can write $V^{\fa,\fd}=\bigsqcup_{x\in X}(x+\fm^{\fa})$ for some finite subset $X\subset V^{\fa,\fd}$. In order to determine the asymptotic behaviour of $\cN_{\phi}(\fb,\fa,\fd,T)$, our next objective will be to find the number of points in $V^{\fa,\fd}\cap\fb^{w}\cap \fB_{\phi}(T)$, as we will shortly argue in the proof of Lemma~\ref{4.6}. To do so we will apply Lemma~\ref{4.4} to translates of $\Lambda=\fm^{\fa}\cap\fb^{w}$, a sublattice of $\fa^{w}$. To accomplish that, write the ideal $\fb$ as $\fb=\fa\fc$, with $\fc\subseteq\cO_{\K}$. Without loss of generality, we may take $\fc$ to be square-free. The reason for this is that we ultimately want to make use of equation~\eqref{equation} to find $N_{\phi}(T)$, where the factor $\mu_{\K}(\fc)$ plays a role, which vanishes on non-square-free ideals.
	
	We denote 
	\begin{equation*}
		\cS:=\{\fp\subset\cO_{\K}\mid \fm_{\fp}^{\fa}\neq(\fa_{\fp})^{w}\};
	\end{equation*}
	this set is finite by Lemma~\ref{congruenceconditions}.
	
	Factor $\fc=\fc_{0}\fc_{1}$ as a product of prime ideals, where $\fc_{0}$ is coprime to all $\fp\in \cS$ and $\fc_{1}$ is a (finite) product of primes $\fp\in \cS$. It follows from co-primality of $\fc_{0}$ and $\fc_{1}$ that $\fc=\fc_{0}\cap\fc_{1}$. 
	
	\begin{notation}
        We write
		\begin{equation*}
			\begin{split}
				&\fm^{\fa}_{\fc_{1}}:=\fm^{\fa}\cap (\fa\fc_{1})^{w};
				\\&V_{\fc_{1}}^{\fa,\fd}:=V^{\fa,\fd}\cap(\fa\fc_{1})^{w}.
			\end{split}
		\end{equation*}
	\end{notation}
	
	We then have $V^{\fa,\fd}\cap\fb^{w}=V_{\fc_{1}}^{\fa,\fd}\cap (\fa\fc_{0})^{w}$. Note that $V_{\fc_{1}}^{\fa,\fd}$ takes only finitely many values by Lemma~\ref{finiteness} and the fact that $\fc_{1}$ by construction only takes $2^{|\cS|}$ values. Further, points $x\in V_{\fc_{1}}^{\fa,\fd}\cap (\fa\fc_{0})^{w}$ satisfy the following congruence conditions: 
	\begin{enumerate}
		\item $x$ lies in the finite set $\fS_{\fc_{1}}^{\fa,\fd}:= V_{\fc_{1}}^{\fa,\fd}/\fm^{\fa}_{\fc_{1}}$ modulo $\fm^{\fa}_{\fc_{1}}$ (recall that $V_{\fc_{1}}^{\fa,\fd}$ is $\fm^{\fa}_{\fc_{1}}$-periodic by Lemma~\ref{vadproperties} (1)),
		\item $x$ is 0 modulo $(\fa\fc_{0})^{w}$.
	\end{enumerate}
	We conclude that $V_{\fc_{1}}^{\fa,\fd}\cap (\fa\fc_{0})^{w}$ is a $((\fa\fc_{0})^{w}\cap \fm^{\fa}_{\fc_{1}})$-periodic subset of $(\fa\fc_{0})^{w}$ and can therefore be written as the union of (finitely many) translates of $(\fa\fc_{0})^{w}\cap \fm^{\fa}_{\fc_{1}}$ in $(\fa\fc_{0})^{w}\subset \fa^{w}$.
	
	We show a version of the Chinese remainder theorem applied to modules.
	
	\begin{lemma}
		\label{CRT}
		Let $M$ be a module over a ring $R$, and $M_{0},M_{1}\subset M$ submodules of $M$ such that $M_{0}+M_{1}=M$. We then get an isomorphism
		\begin{align*}
			M/(M_{0}\cap M_{1}) &\isom M/M_{0}\times M/M_{1}\\
			m\bmod (M_{0}\cap M_{1}) &\longmapsto (m\bmod M_{0},\, m\bmod M_{1}).
		\end{align*}
	\end{lemma}
	
	\begin{proof}
		The map is clearly injective. To prove that it is also surjective, let $a\in M/M_{0}$ and $b\in M/M_{1}$, and write $a=a_{0}+a_{1}$, $b=b_{0}+b_{1}$ with $a_{0},b_{0}\in M_{0}$ and $a_{1},b_{1}\in M_{1}$. Then, the image of $a_{1}+b_{0}\in M$ is $(a_{1},b_{0})$, which is equal to $(a,b)$ in the target space. Hence, the map is surjective.
	\end{proof}
	
	Note that $(\fa\fc_{0})^{w}+\fm^{\fa}_{\fc_{1}}=\fa^{w}$. Applying Lemma~\ref{CRT}, we get the following diagram.
	\begin{equation*}
		\xymatrix{\fa^{w}\ar@{->>}[d]\ar@{->>}[rd]^(.40){\pi} & &\\
			\fa^{w}/((\fa\fc_{0})^{w}\cap \fm^{\fa}_{\fc_{1}})\ar[r]^(.48){\sim} & \fa^{w}/(\fa\fc_{0})^{w}\times \fa^{w}/\fm^{\fa}_{\fc_{1}}&\{0\}\times \fS_{\fc_{1}}^{\fa,\fd} \ar@{^(->}[l]}
	\end{equation*}
	
	\noindent Thus, $V_{\fc_{1}}^{\fa,\fd}\cap (\fa\fc_{0})^{w}=\{m\in \fa^{w}\, |\, \pi(m)\in\{0\}\times \fS_{\fc_{1}}^{\fa,\fd}\}$ can be written as the union of $\#(\{0\}\times \fS_{\fc_{1}}^{\fa,\fd}\})=\#\fS_{\fc_{1}}^{\fa,\fd}$ translates of $(\fa\fc_{0})^{w}\cap \fm^{\fa}_{\fc_{1}}$ inside $\fa^{w}$.
	
	\begin{lemma}
		We have $\det((\fa\fc_{0})^{w}\cap \fm^{\fa}_{\fc_{1}})=N(\fc_{0})^{|w|}(\cO_{\K}^{m}:\fm^{\fa}_{\fc_{1}})\dfrac{D_{\K}^{m/2}}{2^{mr_{2}}}$.
	\end{lemma}
	
	\begin{proof}
		It follows from Lemma~\ref{CRT} that
		\begin{equation*}
			\begin{split}
				\det((\fa\fc_{0})^{w}\cap \fm^{\fa}_{\fc_{1}})&=(\fa^{w}:((\fa\fc_{0})^{w}\cap \fm^{\fa}_{\fc_{1}}))\det(\fa^{w})\\&=(\fa^{w}:(\fa\fc_{0})^{w})(\fa^{w}:\fm^{\fa}_{\fc_{1}})\det(\fa^{w}).
			\end{split}
		\end{equation*}
		Moreover, $\fa^{w}/(\fa\fc_{0})^{w}$ has cardinality $N(\fc_{0})^{|w|}$, since, as a consequence of $\cO_{\K}$ being a Dedekind domain, $(\fa^{w_{i}}:(\fa\fc_{0})^{w_{i}})=(\cO_{\K}:\fc_{0}^{w_{i}})=N(\fc_{0})^{w_{i}}$ for all $i$. Further, $(\fa^{w}:\fm^{\fa}_{\fc_{1}})$ is equal to $(\cO_{\K}^{m}:\fm^{\fa}_{\fc_{1}})/N(\fa)^{|w|}$, and by \cite[Section~4]{De} we know that $\det(\fa^{w})=\frac{D_{\K}^{m/2}N(\fa)^{|w|}}{2^{mr_{2}}}$. This concludes the proof.
	\end{proof}
	
	\begin{lemma}
		\label{v}
		The cardinality of $V^{\fa,\fd}\cap\fb^{w}\cap \fB_{\phi}(T)$ is 
		\begin{equation*}
			\dfrac{\#\fS_{\fc_{1}}^{\fa,\fd}}{N(\fc_{0})^{|w|}(\cO_{\K}^{m}:\fm^{\fa}_{\fc_{1}})}\dfrac{2^{mr_{2}}}{D_{\K}^{m/2}}\vol_{Nm}(\fB_{\phi}(1))T^{\frac{|w|}{e}}+O(T^{\frac{|w|}{e}-\frac{\min\{w_1,\ldots,w_m\}}{eN}}).
		\end{equation*}
	\end{lemma}
	
    \begin{proof}
	This follows from Lemma~\ref{4.4} applied to $\#\fS_{\fc_{1}}^{\fa,\fd}$ translates of the lattice $\Lambda=(\fa\fc_{0})^{w}\cap \fm^{\fa}_{\fc_{1}}$, taking $\fP=\fB_{\phi}(1)\subset\R^{Nm}$ and using Lemma~\ref{bt}. 
	\end{proof}
	
	We derive a formula for $\cN_{\phi}(\fb,\fa,\fd,T)$.
	
	\begin{lemma}
		\label{4.6}
		We have
		\begin{multline*}
			\cN_{\phi}(\fb,\fa,\fd,T)=\dfrac{\#\fS_{\fc_{1}}^{\fa,\fd}}{N(\fc_{0})^{|w|}(\cO_{\K}^{m}:\fm^{\fa}_{\fc_{1}})}\dfrac{1}{\#\mu(\K)}\dfrac{2^{mr_{2}}}{D_{\K}^{m/2}}\vol_{Nm}(\fB_{\phi}(1))T^{\frac{|w|}{e}}\\
                        +O(T^{\frac{|w|}{e}-\frac{\min\{w_1,\ldots,w_m\}}{eN}}N(\fc_{0})^{-|w|+\frac{\min\{w_1,\ldots,w_m\}}{N}}).
		\end{multline*}
	\end{lemma}
	
	\begin{proof}
		By Proposition~\ref{4.2} and Lemma~\ref{v}, we have
		\begin{multline}
			\label{eq}
			\cN_{\phi}(\fb,\fa,\fd,T)=\dfrac{\#\fS_{\fc_{1}}^{\fa,\fd}}{N(\fc_{0})^{|w|}(\cO_{\K}^{m}:\fm^{\fa}_{\fc_{1}})}\dfrac{1}{\#\mu(\K)}\dfrac{2^{mr_{2}}}{D_{\K}^{m/2}}\vol_{Nm}(\fB_{\phi}(1))T^{\frac{|w|}{e}}\\
                        +O(T^{\frac{|w|}{e}-\frac{\min\{w_1,\ldots,w_m\}}{eN}}).
		\end{multline}
		In order to determine the influence of $\fc_{0}$ in the error term, let us write $\fc_{0}=\lambda\fr$ for $\fr\in R_{\K}$ such that $[\fc_{0}]=[\fr]$ and $\lambda\in\K^{\times}$ unique up to units.
		Next, let $x\in\K^{m}-\{0\}$. Note that
		\begin{enumerate}
			\item $\fI_{w}(x)\subseteq\fb\iff\fI_{w}(\lambda_{*}x)\subseteq\lambda\fb$, by Lemma~\ref{scalingidealother} (1);
			\item $S_{u,\infty}(\phi(\lambda_{*}x))=N((\lambda))^{e}S_{u,\infty}(\phi(x))$, since 
			\begin{equation*}
				S_{u,\infty}(\phi(\lambda_{*}x))=S_{u,\infty}(\lambda^{e}_{*}\phi(x))=\prod_{v\in\Omega_{\K}}|\lambda|_{v}^{e} S_{u,\infty}(\phi(x))
			\end{equation*}
			and $\prod_{v\in\Omega_{\K}}|\lambda|_{v}=N((\lambda))$.
		\end{enumerate}
		Consequently, 
		\begin{equation*}
			\begin{split}
				\cN_{\phi}(\fb,\fa,\fd,T)&=\cN_{\phi}(\fc_{0}\fc_{1}\fa,\fa,\fd,T)=\cN_{\phi}(\lambda\fr\fc_{1}\fa,\fa,\fd,T)\\&=\widehat{\#}\{x\in \cO_{\K}^{\times}\backslash (V^{\fa,\fd}-\{0\}) \mid  \fI_{w}(x)\subseteq\lambda\fr\fc_{1}\fa,\, S_{u,\infty}(\phi(x))\leq T\}\\&=\cN_{\phi}(\fr\fc_{1}\fa,\fa,\fd,N((\lambda))^{-e}T),
			\end{split}
		\end{equation*}
		since, as argued before, $\fI_{w}(x)\subseteq\lambda\fr\fc_{1}\fa\iff \fI_{w}(\lambda^{-1}_{*}x)\subseteq\fr\fc_{1}\fa$ and, on the other hand, $S_{u,\infty}(\phi(x))\leq T\iff S_{u,\infty}(\phi(\lambda^{-1}_{*}x))\leq N((\lambda))^{-e}T$.
		
		Applying equation~\eqref{eq} to $\cN_{\phi}(\fr\fc_{1}\fa,\fa,\fd,N((\lambda))^{-e}T)$ yields the result, since $\fc_{0}=\lambda\fr$ by construction.
	\end{proof}
	
	\begin{remark}
		$\vol_{Nm}(\fB_{\phi}(1))$ is the volume of the region defined by some polynomial inequalities determined by the choice of $\phi$. We will not give an explicit expression for this volume, as it might not even be possible to do so for an arbitrary $\phi$. However, this value can be numerically approximated in any given example of $\phi$, as we do in Section~\ref{sec4}.
	\end{remark}
	
	\begin{remark}
		Suppose we had defined our counting functions without the weight $1/\#{\Aut x}$. Then Lemma~\ref{4.6} would remain valid except that $\#\mu(\K)$ should be replaced by $\mu_{\K}^{w} := \frac{\#\mu(\K)}{\gcd\{w_{1},\dots,w_{n},\#\mu(\K)\}}$. Namely, as in \cite[Section~4]{De}, one then needs to count the $\mu(\K)$-orbits in $V^{\fa,\fd}\cap\fb^{w}\cap \fB_{\phi}(T)$, the only difference being that we do not assume $\P(w)$ to be \textit{well-formed} (i.e.\ it does not necessarily hold that each $n-1$ elements of $w$ are coprime). We consequently must replace the factor $w$ in \cite[Proposition~4.2]{De} (denoting the number of roots of unity) by $\mu_{\K}^{w}$ (see \cite[Theorem~3.8]{BrNa}). We do not get a strict equality, since in general only orbits with all coordinates non-zero contain exactly $\mu_{\K}^{w}$ points. However, points with at least one coordinate equal to zero lie on a lower-dimensional subvariety, and by applying the same method used to obtain Lemma~\ref{v} to this subvariety, we deduce that its contribution is at most $O((T^{\frac{1}{eN}})^{|w|N-\min\{w_1,\ldots,w_m\}}) = O(T^{\frac{|w|}{e}-\frac{\min\{w_1,\ldots,w_m\}}{eN}})$.
	    Altogether, this shows that when $\cN_{\phi}(\fb,\fa,\fd,T)$ is defined without the factors $1/\#{\Aut x}$, the quantity $\mu_{\K}^{w} \cdot\cN_{\phi}(\fb,\fa,\fd,T)$ is asymptotically equivalent to the number of points in $V^{\fa,\fd}\cap\fb^{w}\cap \fB_{\phi}(T)$, with discrepancy at most $O(T^{\frac{|w|}{e}-\frac{\min\{w_1,\ldots,w_m\}}{eN}})$.
	\end{remark}

	\subsection{The asymptotic formula for $N_{\phi}(T)$}
	
	We now state our main result, which we will prove by combining the information gained from Lemma~\ref{4.6} and equation~\eqref{equation}.
	
	\begin{theorem}
		\label{maintheorem}
                The asymptotic behaviour of $N_{\phi}(T)$ is
		\begin{multline*}
			N_{\phi}(T) = \left(\dfrac{1}{\zeta_{\K}(|w|)}\left(\frac{2^{r_{2}}}{\sqrt{D_{\K}}}\right)^{m}\dfrac{C_{\phi}^{\K}\vol_{Nm}(\fB_{\phi}(1))}{\#\mu(\K)}\right)T^{\frac{|w|}{e}}\\
                        +O(T^{\frac{|w|}{e}-\frac{\min\{w_1,\ldots,w_m\}}{eN}}),
		\end{multline*}
		where
		\begin{equation*}
			C_{\phi}^{\K}:=\sum_{\fa\in R_{\K}}\sum_{\fd\in\cD_{\phi}}\sum_{\substack{\fc_{1}\subseteq\cO_{\K} \\ \text{product} \\ \text{of }\fp\in\cS}}(N(\fd)^{1/e}N(\fa))^{|w|} \dfrac{\mu_{\K}(\fc_{1})\#\fS_{\fc_{1}}^{\fa,\fd}}{(\cO_{\K}^{m}:\fm^{\fa}_{\fc_{1}})}\prod_{\fp\in\cS}(1-N(\fp)^{-|w|})^{-1}.
		\end{equation*}
		If $\K=\Q$ and $w\in\{(1,1),(2)\}$, the error term is taken to be $O(T^{1/e}\log T^{1/e})$.
	\end{theorem}
	
	\begin{proof}
		We first want to determine a formula for the last factor in equation~\eqref{equation}, namely $\sum_{\fc\subseteq\cO_{\K}}\mu_{\K}(\fc)\cN_{\phi}(\fa\fc,\fa,\fd,TN(\fd)N(\fa)^{e})$. We do so by applying Lemma~\ref{4.6} to $\cN_{\phi}(\fa\fc,\fa,\fd,TN(\fd)N(\fa)^{e})$, and decomposing $\fc$ as above. Note that by doing so, we have $\mu_{\K}(\fc_{0}\fc_{1})=\mu_{\K}(\fc_{0})\mu_{\K}(\fc_{1})$, since $\fc_{0}$ and $\fc_{1}$ are coprime by construction. We have
		\begin{multline*}
			\sum_{\fc\subseteq\cO_{\K}}\mu_{\K}(\fc)\cN_{\phi}(\fa\fc,\fa,\fd,TN(\fd)N(\fa)^{e})\\\sim\sum_{\substack{\fc_{1}\subseteq\cO_{\K} \\ \text{product} \\ \text{of }\fp\in \cS}} \dfrac{\mu_{\K}(\fc_{1})\#\fS_{\fc_{1}}^{\fa,\fd}}{(\cO_{\K}^{m}:\fm^{\fa}_{\fc_{1}})\#\mu(\K)}(TN(\fd)N(\fa)^{e})^{\frac{|w|}{e}}\dfrac{2^{mr_{2}}}{D_{\K}^{m/2}}\vol_{Nm}(\fB_{\phi}(1))\\\times\sum_{\substack{\fc_{0}\subseteq\cO_{\K} \\ \text{product} \\ \text{of }\fp\notin \cS}}\dfrac{\mu_{\K}(\fc_{0})}{N(\fc_{0})^{|w|}}.
		\end{multline*}
		We rewrite the latter sum in the above equation in the following way (recall that $\fc$ is square-free, and therefore so is $\fc_{0}$).
		\begin{equation*}
			\begin{split}
				\sum_{\substack{\fc_{0}\subseteq\cO_{\K} \\ \text{product} \\ \text{of }\fp\notin \cS}}\dfrac{\mu_{\K}(\fc_{0})}{N(\fc_{0})^{|w|}}&=\prod_{\fp\notin \cS}\sum_{k=0}^{1}\dfrac{\mu_{\K}(\fp^{k})}{N(\fp)^{k|w|}}=\prod_{\fp\notin \cS}(1-N(\fp)^{-|w|})\\&=\dfrac{\prod_{\fp\in \cS}(1-N(\fp)^{-|w|})^{-1}}{\prod_{\fp}(1-N(\fp)^{-|w|})^{-1}}=\dfrac{1}{\zeta_{\K}(|w|)}\prod_{\fp\in \cS}(1-N(\fp)^{-|w|})^{-1}.
			\end{split}
		\end{equation*}
		Substituting this into equation~\eqref{equation} and rearranging some of the factors, we get the desired result.
		
		The error term is estimated using \cite{Sc}, where $B=T^{1/e}$, $\fa=\fc_{0}$ and $t=|w|-1/N$ (since there are only finitely many contributions of $\fa$ and $\fd$, we may group these constants together).
	\end{proof}
	
	\begin{remark}
		The expression for $N_{\phi}(T)$ is indeed well defined, since the constant $C_{\phi}^{\K}$ is independent of the choice of the representative $\fa$ in its ideal class. To see this, let $\fb\in\cO_{\K}$ be a principal ideal and note that
		\begin{enumerate}
			\item $N(\fa\fb)^{|w|}=N(\fa)^{|w|}N(\fb)^{|w|}=N(\fa)^{|w|}(\cO_{\K}^{m}:\fb^{w})$;
			\item $(\cO_{\K}^{m}:\fm_{\fc_{1}}^{\fa\fb})=(\cO_{\K}^{m}:\fm^{\fa}_{\fc_{1}})(\cO_{\K}^{m}:\fb^{w})$, since by definition of $\fm^{\fa}$ we have  $\fm^{\fa\fb}=\fm^{\fa}\times \fb^{w}$ (recall by Lemma~\ref{congruenceconditions} that $\fm_{\fp}^{\fa}=\fp^{s_{1}}\fa_{\fp}^{w_{1}}\times\cdots\times\fp^{s_{m}}\fa_{\fp}^{w_{m}}$) and $(\fa\fb\fc_{1})^{w}$ is $(\fa\fc_{1})^{w}\fb^{w}$;
			\item $\#\fS_{\fc_{1}}^{\fa\fb,\fd}=\#\fS_{\fc_{1}}^{\fa,\fd}$, since $V(\fa\fb,\fd)=V^{\fa,\fd}\fb^{w}$ by definition,
		\end{enumerate} 
		which implies that $N(\fa\fb)^{|w|}\frac{\#\fS_{\fc_{1}}^{\fa\fb,\fd}}{(\cO_{\K}^{m}:\fm_{\fc_{1}}^{\fa\fb})}=N(\fa)^{|w|}\frac{\#\fS_{\fc_{1}}^{\fa,\fd}}{(\cO_{\K}^{m}:\fm^{\fa}_{\fc_{1}})}$.
	\end{remark}
	
	It is to be expected that we get the same asymptotic formula as in \cite[Theorem (A)]{De} except for the factor $C_{\phi}^{\K}\vol_{Nm}(\fB_{\phi}(1))$ built up from all these finite contributions from the different $\fa,\fd,\fc_{1}$, which stem from the choice of $\K$ and $\phi$. This factor replaces the volume in \cite[Theorem (A)]{De} and is independent of the choice of the polynomials representing the morphism. Note that we would also have a factor $h$ (denoting the class number of $\K$) in our result, if the summand in $C_{\phi}^{\K}$ were independent of the choice of $\fa\in R_{\K}$.
	
	\section{Examples}
        \label{sec4}
	
	In this section, we give two explicit examples of morphisms between weighted projective lines $\phi:\P(w)\to\P(u)$ and we determine the asymptotic behaviour of the corresponding counting function $N_{\phi}(T)=\widehat{\#}\{x\in\P(w)\mid S_{u}(\phi(x))\leq T\}$ as $T\to\infty$, using the method from Section~\ref{sec3}. For a list of other modular curves for which our method could be applied, see \cite[Table~1]{BrNa}.
	
	\subsection{Elliptic curves with a point of order 2}
	
	In Remark~\ref{elliptic}, we will show that this example allows us to count elliptic curves over $\Q$ with a point of order~2 with respect to a suitable height function.
	
	\subsubsection{The morphism $\phi$}
	
	We define the following map between weighted projective lines:
	\begin{equation}
		\label{phi}
		\begin{aligned}
			\phi: \P(2,4) &\longrightarrow \P(4,6)\\
			(a,b) &\longmapsto (a^2-2b,\, a(3b-a^2)).
		\end{aligned}
	\end{equation}
	By Lemma~\ref{morphism}, the map $\phi$ is a morphism with $e=1$. To check the second condition of Lemma~\ref{morphism}, note that the radical of the ideal $(a^2-2b,\, a(3b-a^2))$ of $\Q[a,b]$ equals $(a,b)$, as is straightforward to verify.
	We can therefore apply the techniques of Section~\ref{sec3} to the map $\phi$, with $m=2$, $w=(2,4)$, $u=(4,6)$ and $\K=\Q$.
	
	\begin{remark}
		\label{elliptic}
		The set $\P(2,4)(\Q)-\{(a,0),(a,3a^2/8)\mid a\in\Q^\times\}$ parametrizes isomorphism classes of elliptic curves over~$\Q$ with a point of order 2. To see this, let $E$ be an elliptic curve over $\Q$ and $P\in E(\Q)$ a point of order 2. We choose a Weierstrass model for~$E$ of the form $y^{2}=x^{3}+(a/4)x^{2}+(b/24)x$ with $P$ having coordinates $(0,0)$. We note that $(E,P)$ determines $(a,b)$ up to multiplication by $\lambda^{2}$ and $\lambda^{4}$ for $\lambda\in\Q^{\times}$, respectively, corresponding to the change of variables $(x,y)=(\lambda^{2}x',\lambda^{3}y')$. For $(E,P)$ as above, the real number $S(E,P):=S_{(2,4)}(a,b)$ is therefore independent of the choice of model. We exclude the points of the form $(a,0)$ and $(a,3a^2/8)$ because the curves defined by the corresponding Weierstrass equations are singular.
		The coordinates of $\phi(a,b)$ correspond to the usual $c_{4}$- and $c_{6}$-invariants of the elliptic curve $E$. Furthermore, in terms of the usual coefficients $a_2$, $a_4$, $b_2$ and $b_4$, we have $a=4a_2=b_2$ and $b=24a_4=12b_4$.

		By definition, we have
                \[
                N_{\phi}(T) = \widehat{\#}\{(a,b)\in\P(2,4)(\Q)\mid S_{(4,6)}(\phi(a,b))\leq T\}.
                \]
                The weighted projective line $\P(2,4)$ has no subvariety of accumulation points with respect to the height function (see \cite[Proposition 6.3]{De} for the definition and the argument, both of which generalize to our situation). Therefore the asymptotic behaviour of $N_{\phi}(T)$ does not change when points of the form $(a,0)$ and $(a,3a^2/8)$ with $a\in\Q^\times$ are omitted. We conclude that the asymptotic behaviour of $N_{\phi}(T)$ is the same as that of the number of isomorphism classes of pairs $(E,P)$ consisting of an elliptic curve over~$\Q$ and a point of order~2 such that $S(E,P)\leq T$.

                Finally, for readers familiar with stacks, we note that by a generalization of the above arguments, the moduli stack $Y_1(2)$ of elliptic curves with a point of order~$2$ can be identified with the open substack of $\P(2,4)$ obtained by omitting points of the form $(a,0)$ or $(a,3a^2/8)$.  More generally, one can show that the compactified moduli stack $X_1(2)$ can be identified with $\P(2,4)$ itself.
	\end{remark}
	
	\subsubsection{The constant $C_{\phi}^{\Q}$}
	
	We will ultimately apply Theorem~\ref{maintheorem} to find an asymptotic formula for $N_{\phi}(T)$. To do so, we first compute the value of $C_{\phi}^{\Q}$. Recall that 
	\begin{equation*}
		C_{\phi}^{\K}:=\sum_{\fa\in R_{\K}}\sum_{\fd\in\cD_{\phi}}\sum_{\substack{\fc_{1}\subseteq\cO_{\K} \\ \text{product} \\ \text{of }\fp\in \cS}}(N(\fd)^{1/e}N(\fa))^{|w|} \dfrac{\mu_{\K}(\fc_{1})\#\fS_{\fc_{1}}^{\fa,\fd}}{(\cO_{\K}^{m}:\fm^{\fa}_{\fc_{1}})}\prod_{\fp\in \cS}(1-N(\fp)^{-|w|})^{-1}.
	\end{equation*}
	
	The group $\Cl(\Q)$ is trivial, and we take $\fa=\Z$ as a representative of the trivial element $c\in \Cl(\Q)$.
	
	\begin{lemma}
		\label{1236}
		Let $(a,b)\in\P(2,4)$. We denote
		\begin{equation*}
			(^*):=\{(a,b)\in\P(2,4): 2^{3}\mid a\text{ and } 2^{3} \mid b \text{ or } 2^{2} \mid a \text{ and } 2^{4} \mid b\}.
		\end{equation*}
		Then,
		\begin{equation*}
			\fI_{u}(\phi(a,b))\cdot\fI_{w}(a,b)^{-1}=
			\begin{cases}
				(2) ,\text{ when } (a,b)\in(^*);\ \\
				(1) ,\text{ else}.\ 
			\end{cases}
		\end{equation*}
	\end{lemma}
	
	\begin{proof}
		Without loss of generality we can take  $\fI_{(2,4)}(a,b)$ to be $(1)$ by multiplying the point $(a,b)$ by $p^{-1}$ (for $p$ a prime number) with the weight $(2,4)$ action if necessary (we say that $(a,b)$ is \textit{primitive}). We now write $\phi(a,b)=(\phi(a,b)_{1},\phi(a,b)_{2})$. We want to find primes $p$ such that $p^{4}\mid\phi(a,b)_{1}=a^{2}-2b$ and $p^{6}\mid\phi(a,b)_{2}=3ab-a^{3}$. We check prime by prime.
		\begin{enumerate}
			\item Let $p=2$. Then, $2^{4}\mid a^{2}-2b\iff 2^{2}\mid a$ and $2^{3}\mid b$. On the other hand, $2^{6}\mid 3ab-a^{3}\iff 2^{2}\mid a$ and $2^{6}\mid ab$. Thus, for $x=(a,b)$ such that either $2^{3}\mid a$ and $2^{3}\mid b$ or $2^{2}\mid a$ and $2^{4}\mid b$, we have $\fI_{(4,6)}(\phi(x))\subseteq(2)$.
			\item Let $p\geq 3$. Then, $p^{4}\mid a^{2}-2b\iff p^{2}\mid a$ and $p^{4}\mid b$.
		\end{enumerate}
		Recall that for the points $x\in\P(2,4)$ for which there does not exist a prime $p$ satisfying $p^{4}\mid \phi(x)_{1}$ and $p^{6}\mid \phi(x)_{2}$, we have $\fI_{(4,6)}(\phi(x))=(1)$. The result follows by primitivity of the point $(a,b)$.
	\end{proof}
	
	\begin{corollary}
		\label{1236S}
		We have $\cD_{\phi}=\{(1),(2)\}$ and $\cS=\{(2)\}$.
	\end{corollary}
	
	\begin{proof}
		Applying Lemma~\ref{1236}, the first statement is obvious. For the second claim, use Lemma~\ref{congruenceconditions} and the definition of the set $\cS$.
	\end{proof}
	
	\begin{lemma}
		We have $\fm^{\fa}=8\Z\times 16\Z$.
	\end{lemma}
	
	\begin{proof}
		By definition, $\fm^{\fa}=\bigcap_{\fp}(\fm_{\fp}^{\fa}\cap\Q^{2})$, where the intersection takes place in $\Q_{\fp}^{2}$. It follows from Lemma~\ref{congruenceconditions} that $\fm_{\fp}^{\fa}=\Z_{\fp}^{2}$ for $\fp\neq (2)$. 
		In order to determine $\fm_{(2)}^{\fa}$, we need to look at the local congruence conditions that $\fd$ being a given value is determined by. Lemma~\ref{1236} then gives us that
		\begin{equation*}
			\fm_{\fp}^{\fa}=
			\begin{cases}
				2^{3}\Z_{(2)}\times 2^{4}\Z _{(2)}=8\Z _{(2)}\times 16\Z _{(2)}, \text{ for } \fp=(2);\ \\
				\Z_{\fp}^{2} ,\text{ for } \fp\neq(2).
			\end{cases}
		\end{equation*}
		From this the claim follows.
	\end{proof}
	
	\begin{corollary}
		As a consequence, 
		\begin{equation*}
			\fm^{\fa}_{\fc_{1}}=8\Z\times 16\Z.
		\end{equation*}
	\end{corollary}
	
	We now describe a method to compute the value of $\#\fS_{\fc_{1}}^{\fa,\fd}$. We do so by studying the case $\fc_{1}=(2)$. The case $\fc_{1}=(1)$ can be treated similarly.
	
	First, note that $S_{\fp}^{\fa,\fd}$ is trivial when $\fp\neq (2)$ by Remark~\ref{rpkpjprpkp}. For $\fp=(2)$, we have the following
	\begin{equation*}
		S_{(2)}^{\Z,\fd}\subseteq (\Z_{(2)}/8\Z_{(2)})\times(\Z_{(2)}/16\Z_{(2)}), \text{ where } |S_{(2)}^{\Z,(2)}|=3.
	\end{equation*}
	This is easy to see by, for instance, writing a matrix with the different values of $(a,b)$ modulo $\fm_{(2)}^{\fa}$, where in the position corresponding to $(a,b)$ we get the corresponding value for the local discrepancy $j_{(2)}$, which can either be 0 or 1. We get the following matrix:
	\[
	\begin{pmatrix}
			1&0&0&0&0&0&0&0&1&0&0&0&0&0&0&0\\
			0&0&0&0&0&0&0&0&0&0&0&0&0&0&0&0\\
			0&0&0&0&0&0&0&0&0&0&0&0&0&0&0&0\\
			0&0&0&0&0&0&0&0&0&0&0&0&0&0&0&0\\
			1&0&0&0&0&0&0&0&0&0&0&0&0&0&0&0\\
			0&0&0&0&0&0&0&0&0&0&0&0&0&0&0&0\\
			0&0&0&0&0&0&0&0&0&0&0&0&0&0&0&0\\
			0&0&0&0&0&0&0&0&0&0&0&0&0&0&0&0\\
	\end{pmatrix}.
	\]
	We first look for points $(a,b)\in 2^{2}\Z\times 2^{4}\Z$ with non-zero local discrepancy. We see that we only get $j_{(2)}=1$ for the positions $(a,b)$ in $\{(0,0), (4,0)\}$. On the other hand, we notice that by Lemma~\ref{1236} there is no point $(a,b)$ in $2^{2}\Z\times 2^{4}\Z$ with $j_{(2)}=0$.
	
	Hence, we get
	\begin{equation*}
		\#\fS_{(2)}^{\fa,\fd}=
		\begin{cases}
			0, \text{ for } \fd=(1);\ \\
			2, \text{ for } \fd=(2).
		\end{cases}
	\end{equation*}
	
	See Table 1 for an overview of the elements appearing in the expression for $C_{\phi}^{\Q}$. 
	Using a computer algebra system, we conclude the following.
	
	\begin{lemma}
		\label{ck}
		Notations as above. We have
		\begin{equation*}
			C_{\phi}^{\Q}=3/2.
		\end{equation*}
	\end{lemma}
	\begin{table}[H]
		\centering
		\begin{center}
			\label{tab:table}
			\begin{tabular}{cc}
				\toprule 
				\textbf{Factors in $C_{\phi}^{\K}$} & \textbf{Value for $\phi$}\\
				\midrule 
				$\vphantom{\Big|}R_{\K}$ & $\{\Z\}$\\
				\cline{1-2}
				$\vphantom{\Big|}\cD_{\phi}$ & $\{(1),(2)\}$ \\
				\cline{1-2}
				\vphantom{\Big|}$\cS$ & $\{(2)\}$\\
				\cline{1-2}
				\vphantom{\Big|}$(N(\fd)^{1/e}N(\fa))^{|w|}$ & 1, for $\fd=(1)$ \\
				& 64, for $\fd=(2)$\\
				\cline{1-2}
				\vphantom{\Big|}$\mu_{\K}(\fc_{1})$ & 1, for $\fc_{1}=(1)$ \\
				& $-1$, for $\fc_{1}=(2)$ \\
				\cline{1-2}
				\vphantom{\Big|}$\#\fS_{\fc_{1}}^{\fa,\fd}$ & 125, for $\fd=(1),\fc_{1}=(1)$ \\
				& 0, for $\fd=(1),\fc_{1}=(2)$ \\
				& 3, for $\fd=(2),\fc_{1}=(1)$ \\
				& 2, for $\fd=(2),\fc_{1}=(2)$ \\
				\cline{1-2}
				\vphantom{\Big|}$(\cO_{\K}^{m}:\fm^{\fa}_{\fc_{1}})$ & 128 \\
				\cline{1-2}
				\vphantom{\Big|}$\prod_{\fp\in\cS}(1-N(\fp)^{-|w|})^{-1}$ & 64/63 \\ 
				\bottomrule 
			\end{tabular}
			\caption{\textit{Values of the different factors appearing in the expression for $C_{\phi}^{\K}$ applied to $\phi$ in equation~\eqref{phi}.}}
		\end{center}
	\end{table}
	
	\subsubsection{The asymptotic formula for $N_{\phi}(T)$}
	
        Recall that $\fB_{\phi}(1)=\fD_{\phi}(1)\cap\Delta$, where
        \begin{equation*}
                \fD_{\phi}(1)=\{(a,b)\in\R^{2}-\{0\}\mid \max\{|a^{2}-2b|^{1/4},|a(3b-a^{2})|^{1/6}\}\leq 1\}
        \end{equation*}
        and $\Delta=({\pr}\circ\eta)^{-1}F$ (see Subsection~\ref{lattice-point-problem}).
        The region $\fD_{\phi}(1)$ is shown in Figure~\ref{figure1}.
        \begin{figure}
        \begin{center}
        \includegraphics[scale=0.8]{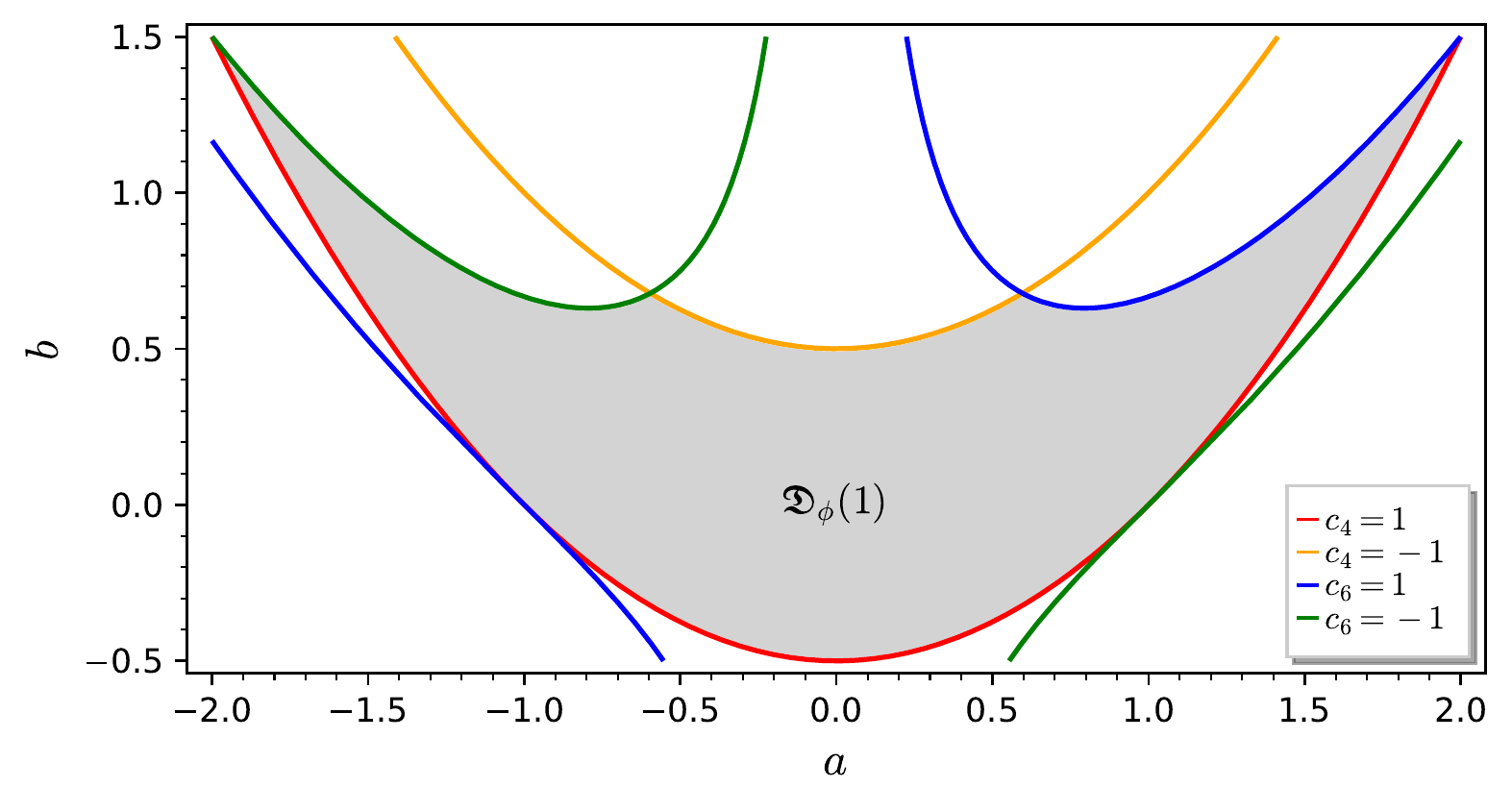}
        \end{center}
        \caption{The region $\fD_{\phi}(1)$ for $X_1(2)$}
        \label{figure1}
        \end{figure}

	\begin{lemma}
		The volume of the region $\fB_{\phi}(1)$ is
		\begin{equation*}
			\vol_{2}(\fB_{\phi}(1))=\dfrac{4+2\log 2+2\alpha-2\log\alpha}{3} = 2.53774\ldots,
		\end{equation*}
		where $\alpha=0.59607\ldots$ is the unique real root of the polynomial $x^{3}+3x-2$.
	\end{lemma}
	
	\begin{proof}
        In this setting, the hyperplane $\cH$ is a single point and $\Gamma$ has rank 0, so $F=\cH$ is a fundamental domain for $\cH/\Gamma$ and $\Delta$ is the entirety of $\R^{2}-\{0\}$. Thus we have $\fB_{\phi}(1)=\fD_{\phi}(1)$.
		By symmetry we have 
		\begin{equation*}
			\vol_{2}(\fB_{\phi}(1))=2\left(\int_{0}^{\alpha} \left(\dfrac{1+x^{2}}{2}-\dfrac{x^{2}-1}{2}\right) \,dx+\int_{\alpha}^{2} \left(\dfrac{x^{3}+1}{3x}-\dfrac{x^{2}-1}{2}\right) \,dx\right).
		\end{equation*}
		Integrating this expression and using $\alpha^{3}=-3\alpha+2$, we get the desired result.
	\end{proof}
	
	See Table 2 for an overview of the elements appearing in the expression for $N_{\phi}(T)$. 
	\begin{table}[H]
		\begin{center}
			\label{tab:table2}
			\begin{tabular}{cc}
				\toprule 
				\textbf{Factors in $N_{\phi}(T)$} & \textbf{Value for $\phi$}\\
				\midrule 
				\vphantom{\Big|}$C_{\phi}^{\K}$ & 3/2 \\
				\cline{1-2}
				\vphantom{\Bigg|}$\zeta_{\K}(|w|)$ & $\dfrac{\pi^{6}}{945}$\\
				\cline{1-2}
				\vphantom{\Big|}$r_{2}$ & 0 \\
				\cline{1-2}
				\vphantom{\Big|}$D_{\K}$ & 1\\
				\cline{1-2}
				\vphantom{\Bigg|}$\vol_{Nm}(\fB_{\phi}(1))$ & $\dfrac{4+2\log 2+2\alpha-2\log\alpha}{3}$ \\
				\cline{1-2}
				\vphantom{\Big|}$\#\mu(\K)$ & 2\\
				\bottomrule 
			\end{tabular}
			\caption{\textit{Values of the different factors appearing in the expression for $N_{\phi}(T)$ applied to $\phi$ in equation~\eqref{phi}.}}
		\end{center}
	\end{table}
	Applying Theorem~\ref{maintheorem} and using a computer algebra system, we conclude the following.
	
	\begin{theorem}
		Notations as above. We have 
		\begin{equation*}
			N_{\phi}(T)= C\,T^{6} + O(T^{5}),
		\end{equation*}
		with 
		\begin{equation*}
			C=\dfrac{945}{2\cdot\pi^{6}}(2+\log 2+\alpha-\log\alpha)= 1.87086\ldots
		\end{equation*}
	\end{theorem}
	
	\subsection{Elliptic curves with a point of order 3}

    The modular curve $X_1(3)$ over~$\Q$ is isomorphic to the weighted projective line $\P(1,3)$: an elliptic curve $E$ with a point $P$ of order~$3$ over a field~$\K$ of characteristic~$0$ has a Weierstrass equation of the form
    \[
    E: y^2 + axy + \frac{b}{6}y = x^3
    \]
    with $P=(0,0)$.  The point $(a,b)\in\P(1,3)(\K)$ is independent of the choice of Weierstrass model as above.  In terms of the coordinates $(a,b)$ on $X_1(3)$ and $(c_4,c_6)$ on $X(1)$, the canonical morphism $X_1(3)\to X(1)$ corresponds to the morphism
    \begin{align*}
    \phi:\P(1,3)&\longrightarrow\P(4,6)\\
    (a,b)&\longmapsto(a^4 - 4ab, -a^6 + 6a^3b - 6b^2).
    \end{align*}
    We can therefore apply Theorem~\ref{maintheorem} with $m=2$, $w=(1,3)$, $u=(4,6)$ and $\K=\Q$.
    In Theorem~\ref{X1_3} below, we will give an asymptotic expression for $N_\phi(T)$.

    As before, we take $R_\Q=\{(1)\}$.
    
    \begin{lemma}
    The set $\cD_\phi$ equals $\{(1)\}$.
    \end{lemma}

    \begin{proof}
    Let $x\in\P(1,3)(\Q)$ be given. We may assume that $x$ is represented by a pair $(a,b)\in\Z^2$ such that there exists no prime number~$p$ with $p\mid a$ and $p^3\mid b$.  Then we have $\fI_{(1,3)}(a,b)=(1)$. Suppose that for some prime number~$p$ we have $p^4\mid a^4-4ab$ and $p^6\mid -a^6+6a^3b-6b^2$. If $p\nmid a$, then we obtain $p^4\mid a^3-4b$, hence $p^4\mid -(4b)^2+6(4b)b-6b^2=2b^2$, which implies $p^2\mid b$, but since $p^4\mid a^3-4b$ this gives $p\mid a$, contradiction.  Thus we have $p\mid a$, implying $p^6\mid 6b(a^3-b)$ and hence $p^3\mid b$, which together with $p\mid a$ gives a contradiction. Thus such a $p$ does not exist, so we have $\fI_{(4,6)}(\phi(a,b))=(1)$.
    \end{proof}

    \begin{corollary}
    We have $C^\Q_\phi=1$.
    \end{corollary}
    
    The region $\fD_\phi(1)$ is defined by the inequalities
    \[
    -1\le a^4-4ab\le 1
    \quad\text{and}\quad
    -1\le-a^6+6a^3b-6b^2\le1.
    \]
    This region is shown in Figure~\ref{figure2}. As before, we have $\fB_\phi(1)=\fD_\phi(1)$.

    \begin{figure}
    \begin{center}
    \includegraphics[scale=0.8]{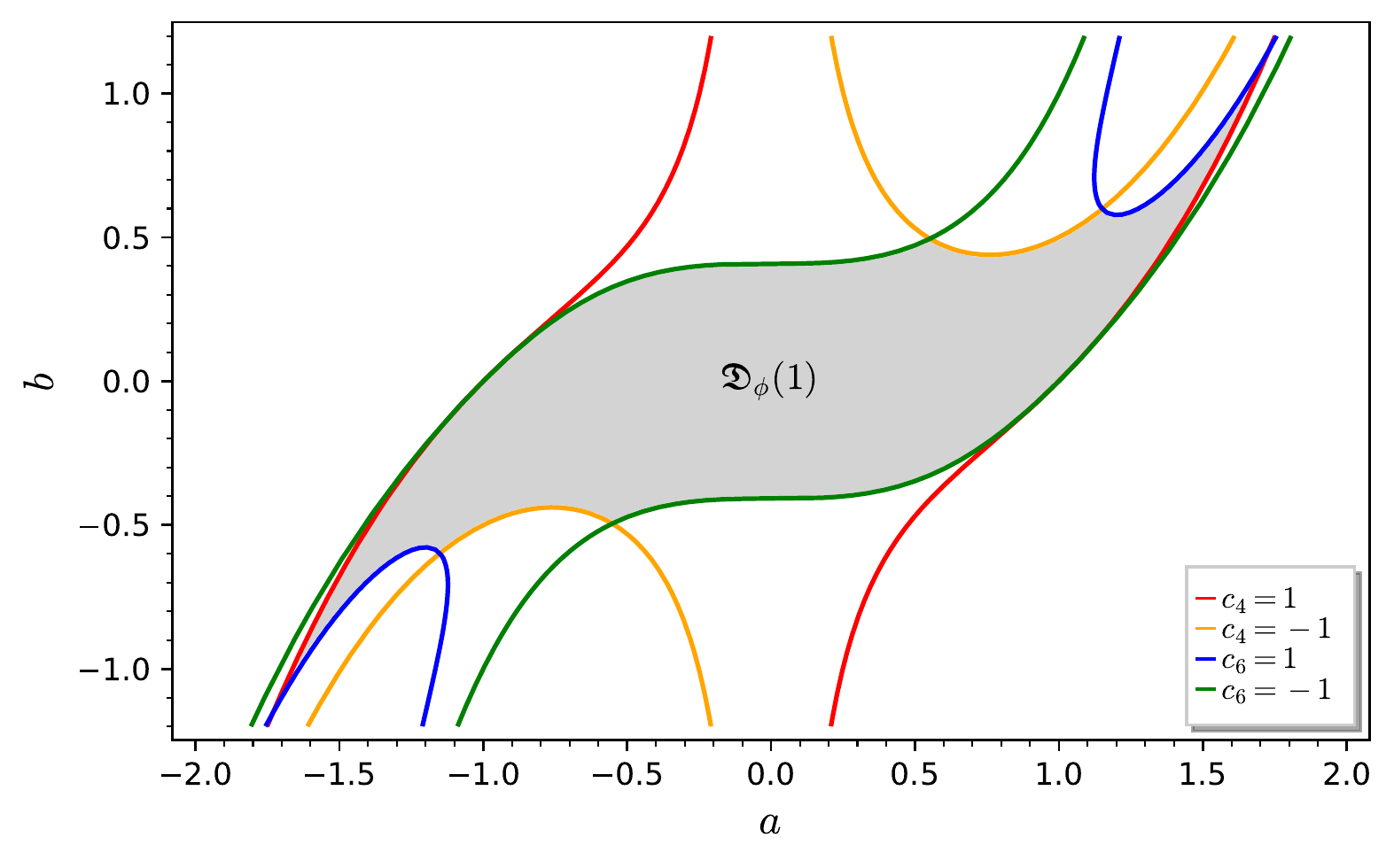}
    \end{center}
    \caption{The region $\fD_{\phi}(1)$ for $X_1(3)$}
    \label{figure2}
    \end{figure}

    \begin{lemma}
    The volume of the region $\fB_\phi(1)$ is
    \begin{align*}
    \vol_2(\fB_\phi(1)) &= 1+I_0-I_1-\frac{\alpha_1^2-\alpha_0^2}{8}+\frac{1}{4}\log(3\alpha_1/\alpha_0)\\
    &=1.8217\ldots,
    \end{align*}
    where $-\alpha_0=-0.3044\ldots$ and $\alpha_1=1.3240\ldots$ are the unique negative and positive real root of the polynomial $x^4+6x^2-8x-3$, respectively, and
    \begin{align*}
    I_0&=\frac{1}{\sqrt{3}}\int_{-\sqrt{\alpha_0}}^1 \sqrt{a^6+2}\,da,\\
    I_1&=\frac{1}{\sqrt{3}}\int_{\sqrt{\alpha_1}}^{\sqrt{3}}\sqrt{a^6-2}\,da.
    \end{align*}
    \end{lemma}

    \begin{proof}
    The inequalities defining $\fD_\phi(1)$ are equivalent to
    \begin{align*}
    -\sqrt{3}&\le a\le\sqrt{3},\\
    b&\le\begin{cases}
    \frac{a^4-1}{4a}& \text{if }-\sqrt{3}\le a\le -1,\\
    \frac{a^3}{2}+\sqrt{\frac{a^6+2}{12}}& \text{if }-1\le a\le\sqrt{\alpha_0},\\
    \frac{a^4+1}{4a}& \text{if }\sqrt{\alpha_0}\le a\le\sqrt{\alpha_1},\\
    \frac{a^3}{2}-\sqrt{\frac{a^6-2}{12}}& \text{if }\sqrt{\alpha_1}\le a\le\sqrt{3},
    \end{cases}\\
    b&\ge\begin{cases}
    \frac{a^3}{2}+\sqrt{\frac{a^6-2}{12}}& \text{if }-\sqrt{3}\le a\le-\sqrt{\alpha_1},\\
    \frac{a^4+1}{4a}& \text{if }-\sqrt{\alpha_1}\le a\le -\sqrt{\alpha_0},\\
    \frac{a^3}{2}-\sqrt{\frac{a^6+2}{12}}& \text{if }-\sqrt{\alpha_0}\le a\le 1,\\
    \frac{a^4-1}{4a}& \text{if }1\le a\le\sqrt{3}.
    \end{cases}
    \end{align*}
    It is now straightforward to express the volume as a finite sum of integrals over~$a$ and to simplify this to the expression given in the lemma.
    \end{proof}

    Again applying Theorem~\ref{maintheorem} and using the identity $\zeta(4)=\pi^4/90$, we obtain the following explicit counting result.

    \begin{theorem}
    \label{X1_3}
    Notations as above. We have
    \[
    N_\phi(T) = CT^4 + O(T^3),
    \]
    where
    \[
    C = \frac{45}{\pi^4}\vol_2(\fB_\phi(1)) = 0.8416\ldots
    \]
    \end{theorem}

	
	\section{Related results and future work}
        \label{sec5}
	
        Theorem~\ref{maintheorem} first appeared in the second-named author's master's thesis \cite{MA}, which was supervised by the first-named author.
	While this paper was being written, Tristan Phillips \cite[Theorem~1.2.1]{Ph} proved a simultaneous generalization of our Theorem~\ref{maintheorem} and a result of Bright, Browning and Loughran \cite[\S3]{BBL}, who refined Schanuel's theorem by allowing local conditions to be imposed at infinitely many places.  Phillips also gives applications to modular curves, providing counting results for modular curves isomorphic to either a weighted projective space or $\P(1)\times\P(2)$ \cite[Theorems 1.1.1 and 1.1.2]{Ph}.
	
	Turning to possible future work, it would be interesting to obtain an asymptotic formula for $N_{\phi}(T)$ for $\phi$ a morphism between spaces that are `close' to being weighted projective spaces. An example of this would be the canonical morphism $X_0(3)\to X(1)$; the modular curve $X_0(3)$ has coarse moduli space $\P^1$, but is not a weighted projective line \cite[Remark 7.4]{BrNa}. In \cite{PiPoVo}, an asymptotic formula is given for counting $\Q$-points of bounded height on $X_{0}(3)$ using a method entirely different from ours. It would be interesting to find out if our methods can be extended to their setting.
	
	On top of that, a (simpler) generalization would be to prove a result analogous to Theorem~\ref{maintheorem} for morphisms between products of weighted projective spaces.
	
        \subsubsection*{Acknowledgements}
        We would like to thank Tristan Phillips for several useful comments and corrections, and for sharing an early version of~\cite{Ph} with us.

\end{document}